\documentclass[11pt,a4paper]{article}

\usepackage{showlabels}

\usepackage{amsmath,amsfonts,amssymb,amsthm}
\usepackage{graphics,epsfig,color,enumerate}
\usepackage{setspace,mathtools,enumitem,dsfont,verbatim}

\usepackage{todonotes}
\usepackage[english]{babel}
\usepackage[utf8]{inputenc}
\usepackage[T1]{fontenc}
\usepackage{csquotes}

\usepackage[nocompress]{cite}
\usepackage[hidelinks]{hyperref}

\usepackage{subcaption}

\usepackage{letltxmacro}
\newcommand{\reffont}{\rm} 
\AtBeginDocument{%
\LetLtxMacro\oldref\ref
\renewcommand{\ref}[1]{%
{\reffont\oldref{#1}}}
}

\makeatletter
\newcommand{\customitem}[1]{%
\item[\rm#1]\protected@edef\@currentlabel{#1}%
}
\makeatother

\makeatletter
\def\@captype{figure}
\@addtoreset{equation}{section}
\makeatother

\newcounter{extralabel}[section]

\newtheoremstyle{fancyremark}{\topsep}{\topsep}{\rm}{}{\bfseries}{.}{ }{}
\newtheoremstyle{fancybreak}{\topsep}{\topsep}{\itshape}{}{\bfseries}{.}{\newline}{}

\theoremstyle{fancybreak}
\newtheorem{ittheorem}{Theorem}
\newtheorem{itlemma}{Lemma}
\newtheorem{itproposition}{Proposition}
\newtheorem{itcorollary}{Corollary}
\newtheorem{italgorithm}{Algorithm}
\newtheorem{itassumption}{Assumption}

\theoremstyle{plain}
\newtheorem{itcondition}{Condition}
 
\theoremstyle{fancyremark}
\newtheorem{itremark}{Remark}
 
\theoremstyle{definition}
\newtheorem{itdefinition}{Definition}

 \newenvironment{theorem}{\addtocounter{extralabel}{1}
 \begin{ittheorem}}{\end{ittheorem}}

 \newenvironment{lemma}{\addtocounter{extralabel}{1}
 \begin{itlemma}}{\end{itlemma}}

 \newenvironment{proposition}{\addtocounter{extralabel}{1}
 \begin{itproposition}}{\end{itproposition}}

 \newenvironment{definition}{\addtocounter{extralabel}{1}
 \begin{itdefinition}}{\end{itdefinition}}

 \newenvironment{corollary}{\addtocounter{extralabel}{1}
 \begin{itcorollary}}{\end{itcorollary}}

\newcommand\xqed[1]{
\leavevmode\unskip\penalty9999 \hbox{}\nobreak\hfill
\quad\hbox{#1}}
\newcommand\RemarkEnd{\xqed{$\blacksquare$}}
\newenvironment{remark}{\addtocounter{extralabel}{1}

\begin{itremark}}{\RemarkEnd\end{itremark}}

\newenvironment{condition}{\addtocounter{extralabel}{1}
\begin{itcondition}}{\end{itcondition}}

\newcommand{\N}{\mathbb{N}}
\newcommand{\NZ}{\mathbb{N}_0}

\DeclareTextFontCommand{\alwaysitalic}{\normalfont\itshape}
\newcommand{\Conf}{\text{\alwaysitalic{Conf}}}

\newcommand{\UConf}{U_{\Conf_H}}
\newcommand{\sss}{\scriptscriptstyle}
\newcommand{\TV}{\sss \mathrm{TV}}
\newcommand{\vep}{\varepsilon}

\newcommand{\tmixstat}{t_{\mathrm{mix}}^{\mathrm{stat}}}

\newcommand{\eqn}[1]{\begin{equation} #1 \end{equation}}
\newcommand{\eqan}[1]{\begin{align} #1 \end{align}}
\newcommand{\nn}{\nonumber}
\newcommand{\expec}{\mathbb{E}}
\newcommand{\prob}{\mathbb{P}}

\newcommand{\sizeH}{|H|}

\newcommand{\dmax}{d_\mathrm{max}}
\newcommand{\whp}{\mathrm{whp}}

\newcommand{\cstat}{c_{n}^{\mathrm{stat}}}

\renewcommand{\pmod}{\mathbb{P}^{\mathrm{mod}}}

\newcommand{\pcouple}{\mathbb{P}^{\mathrm{couple}}}

\newcommand{\D}{\mathcal{D}}

\newcommand{\Dstat}{\mathcal{D}^{\mathrm{stat}}}
\newcommand{\Ddyn}{\mathcal{D}^{\mathrm{dyn}}}

\newcommand{\allx}{x_{[0,t-1]},\allowbreak \bar{x}_{[0,t-1]},\allowbreak \hat{x}_{[r]},\allowbreak \tilde{x}_{[r]}}
\newcommand{\Tallx}{T,\allowbreak \allx}
\newcommand{\eventDSA}{\mathsf{DSA}(\Tallx)}
\newcommand{\eventDSAy}{\mathsf{DSA}(T, y_{[0,t-1]}, \bar{y}_{[0,t-1]}, \hat{y}_{[r]}, \tilde{y}_{[r]})}
\newcommand{\ally}{y_{[0,t-1]}, \bar{y}_{[0,t-1]}, \hat{y}_{[r]}, \tilde{y}_{[r]}}
\newcommand{\eventSA}[1]{\mathsf{SA}(#1)}

\newcommand{\Hdeg}{\deg_H}

\newcommand{\localset}[2]{\mathrm{Local}_{#1}(#2) }
\newcommand{\nearset}[3]{\mathrm{Near}_{#1,#3}(#2) }

\newcommand{\maxHdeg}{(\log{n})^{2+\varepsilon}}

\renewcommand{\theenumii}{(\alph{enumii})}
\renewcommand{\labelenumii}{\theenumii}
\makeatletter
\renewcommand\p@enumiii{\theenumi\theenumii}
\renewcommand\p@enumii{\theenumi}
\makeatother

\title{Linking the mixing times of random walks\\ 
on static and dynamic random graphs}

\author{

Luca Avena
\footnotemark[1]
\\

Hakan G\"{u}lda\c{s}
\footnotemark[1]
\\

Remco van der Hofstad
\footnotemark[2]
\\

Frank den Hollander
\footnotemark[1]
\\

Oliver Nagy
\footnotemark[1]
}

\footnotetext[1]{
Mathematical Institute, Leiden University, 
P.O.\ Box 9512, 2300 RA Leiden, The Netherlands
}

\footnotetext[2]{
Department of Mathematics and Computer Science, Eindhoven University of Technology, 
P.O.\ Box 513, 5600 MB Eindhoven, The Netherlands
}

\date{February 7, 2021}

\begin{document}

\maketitle

\begin{abstract}
This paper considers non-backtracking random walks on random graphs generated according to the configuration model. The quantity of interest is the scaling of the mixing time of the random walk as the number of vertices of the random graph tends to infinity. Subject to mild general conditions, we link two mixing times: one for a \emph{static} version of the random graph, the other for a class of \emph{dynamic} versions of the random graph in which the edges are randomly \emph{rewired} but the degrees are preserved. The link is provided by the probability that the random walk has not yet stepped along a previously rewired edge. We use this link to compute the scaling of the mixing time for three specific classes of random rewirings. Depending on the speed and the range of the rewiring relative to the current location of the random walk, the mixing time may exhibit no cut-off, one-sided cut-off or two-sided cut-off, a trichotomy that was also found in earlier work. Interestingly, for a class of dynamics that are `mesoscopic', i.e., non-local and non-global, we find new behaviour with six subregimes. Proofs are built on a new and flexible coupling scheme, in combination with sharp estimates on the degrees encountered by the random walk in the static and the dynamic version of the random graph. Some of these estimates require sharp control on possible short-cuts in the graph between the edges that are traversed by the random walk. 

\vspace{0.5cm}
\noindent
\small
\emph{Key words.}
Configuration model, random rewiring, random walk, mixing time, cutoff.\\
\emph{MSC2010.} 
05C81, 
37A25, 
60K37, 
82C27. 
\\
\noindent
\emph{Acknowledgment.}
The work in this paper was supported by the Netherlands Organisation for Scientific Research (NWO) through Gravitation-grant NETWORKS-024.002.003. 
\normalsize
\end{abstract}

\newpage 


\section{Introduction}
\label{sec:intro}

\paragraph{Target.}
In the present paper we study the mixing time of a \emph{non-backtracking} random walk on a \emph{dynamically rewired random graph} initially drawn according to the \emph{configuration model}. Our core result is a link between the mixing times on the static and the dynamic random graph. Subject to mild conditions on the degrees of the vertices and the dynamics of the underlying graph, we show that, up to an error that vanishes as the number of vertices tends to infinity, the total variation distance to the stationary distribution on the \emph{dynamic} random graph is given by the total variation distance on the \emph{static} random graph multiplied by the probability that the random walk has not yet stepped along a previously rewired edge. Phrased in symbols, we show that
\begin{equation}
\label{mixlink}
\Ddyn_{x,\xi}(t) = \prob_{x,\xi}(\tau > t)\,\Dstat_{x,\xi}(t) +o_{\sss\prob}(1),
\end{equation}
where $x$ is the starting vertex of the random walk, $\xi$ is the starting configuration of the random graph, $\Ddyn_{x,\xi}(t)$ and $\Dstat_{x,\xi}(t)$ are the total variation distance between the distribution of the random walk at time $t$ and the stationary distribution for the dynamic, respectively, the static random graph, and $\tau$ is the first time the random walk crosses a rewired edge (see Theorem~\ref{thm:main} below for a precise statement). The latter acts as a randomised stopping time and plays a central role in our analysis. 

\paragraph{Innovative aspects.}
Our goal is to build a \emph{general framework} that can be applied to a large class of random graph dynamics for which the degree structure is preserved, including dynamics that depend on the \emph{position} of the random walk and dynamics that are \emph{non-Markovian}. To do so we use a coupling that works well for non-backtracking random walks. We show that \eqref{mixlink} holds under \emph{general conditions} that appear to be the weakest possible, and that can be verified in specific examples. In particular, we use \eqref{mixlink} to identify the scaling of the random walk mixing time for three choices of the dynamics where the rewiring is done in a certain range around the current position of the random walk. Depending on the speed and the range of the rewiring, the mixing time may exhibit no cut-off, one-sided cut-off or two-sided cut-off, a trichotomy that was also found in earlier work (see Section~\ref{sec:literature} for an extensive literature overview). Interestingly, for a class of dynamics that are non-local \emph{and} non-global, we find new behaviour with six subregimes (see Fig.~\ref{fig:tvdisttrichotomy-ntg} below), two of which include \emph{critical crossover times} where the mixing profile \emph{changes shape}.


\subsection{Model and notation}
\label{sec:intro:modelnotation}

It is convenient to describe our model in terms of \emph{half-edges}. Write $V$ to denote the vertex set of the graph, $|V| \eqqcolon n$ the number of vertices, and $\deg(v)$ the degree of vertex $v\in V$. To each vertex $v\in V$ we associate $\deg(v)$ half-edges, forming the set $H_v \coloneqq \{h_i\}_{i=1}^{\deg(v)}$. The set of all half-edges is $H \coloneqq \bigcup_{v \in V} H_v$. We denote the vertex $v$ for which $h\in H_v$ by $v(h) \in V$. If $x,y\in H_v,\ x\neq y$, then we write $x\sim y$ and say that $x$ and $y$ are \emph{siblings} of each other. Using $|X|$ to denote the cardinality of set $X$, we define the degree of a half-edge $h\in H$ as
\begin{equation}
\label{degdef3}
\Hdeg(h) \coloneqq |\{h_i \in H_{v(h)}\colon\, h_i \sim h\}| = \deg(v(h))-1.
\end{equation}

We identify an edge with a pair of half-edges. A \emph{configuration} is a pairing $\xi$ of half-edges with the property that $\xi(h) \neq h$ and $\xi(\xi(h)) = h$ for all $h\in H$. The set of all configurations on $H$ is denoted by $\Conf_H$, and the uniform distribution on $\Conf_H$ is denoted by $\UConf$. With a slight abuse of notation, we will use the same symbol $\xi$ to denote the set of pairs of half-edges forming $\xi$, so $\{x,y\}\in\xi$ means that $\xi(x) = y$ and $\xi(y) = x$. Note that $\xi$ may represent a multi-graph, possibly with self-loops. A random graph corresponding to a configuration where the half-edges are paired uniformly at random is called the \emph{configuration model} (see \cite{BOL1980}, \cite[Chapter 7]{vdH2016}). The quantities above depend on $n$, but this dependence will be mostly suppressed from the notation.

We study Markov chains $\{(X_t,C_t)\}_{t \in \NZ}$, where $X_t\in H$ denotes the \emph{non-backtracking} random walk component and $C_t\in\Conf_H$ corresponds to the \emph{evolution of the underlying graph}. The evolution is chosen in such a way that it does not change the degree sequence of the graph (and consequently does not change the stationary distribution of the random walk on the graph), and can be visualised by breaking up pairs of half-edges and pairing them again, both according to prescribed rules. At each time $t\in\N$, we first update the configuration and then let the walk move.

\begin{remark}[Notation]
Note that $(X_{t-1},C_{t-1})$ is the state just before the transition at time $t$, while $(X_t,C_t)$ is the state just after the transition at time $t$.
\end{remark}

Our main result concerns the total variation distance between the distribution of the random walk component and the stationary uniform distribution on the set of half-edges $U_H$, defined as 
\begin{equation}
\Ddyn_{x,\xi}(t) \coloneqq \|\prob_{x,\xi}(X_t\in\cdot) - U_H(\cdot)\|_{\TV}.
\end{equation}

Here, the total variation distance between two probability measures $\mu$ and $\nu$ on the same finite state space $S$ is defined by
\begin{equation}
\|\mu-\nu\|_{{\sss \mathrm{TV}}} \coloneqq \sum_{x\in S}|\mu(x)-\nu(x)| 
= \sum_{x\in S}[\mu(x)-\nu(x)]_+ = \sup_{A\subseteq S}[\mu(A)-\nu(A)],
\end{equation}
We are concerned with the behaviour of $\D_{x,\xi}(t)$ for \enquote{typical} choices of $x$ and $\xi$. We formalise the notion of typicality in the following definition:

\begin{definition}[With high probability]
\label{def:whp3}
Recall that $n = |V|$ and let $\mu \coloneqq U_H\times\UConf$. A~statement that depends on the initial half-edge $x$ and the initial configuration $\xi$ is said to hold \emph{with high probability}, abbreviated $\whp$, if the $\mu$-measure of the set of pairs $(x,\xi)$ for which the statement holds tends to $1$ as $n\to\infty$.
\end{definition}

Another important object is the first time the random walk steps along a previously rewired edge:

\begin{definition}[Randomized stopping time]\label{def:stopping}
Let $R_t$ be the set of edges being rewired at time~$t$, $R_{\leq t} \coloneqq \bigcup_{s=1}^t R_s$, and let $I_t$ denote the indicator of the event that the random walk steps along a previously rewired edge at time $t$, i.e., $I_t = 1$ when $X_{t-1} \in R_{\leq t}$ and $I_t = 0$ otherwise. We define the \emph{randomized stopping time} $\tau$ as
\begin{equation}
\tau \coloneqq \min\{t\in\N\colon\, I_t = 1\}.
\end{equation}
\end{definition}

Note that, since rewiring happens before the random walk steps, $X_{t-1}$ is the position of the random walk just before it steps over an edge that is rewired at time $t$.

For $x\in H$ and $\xi\in\Conf_H$, we denote by $\Dstat_{x,\xi}(t)$ the total variation distance of the random walk on the static random graph to the stationary uniform distribution $U_H$ at time $t$, and by $\prob_{x,\xi}(\tau > t)$ the probability that $\tau > t$, both given the starting state $(x, \xi)$. 


\subsection{Mixing for general rewiring mechanisms}\label{sec:intro:mainthm}

The main theorem of this paper is the following statement linking the total variation distance to the stationary distribution for the static and the dynamic version of the random graph:

\begin{theorem}[Link between static and dynamic mixing]
\label{thm:main}
Suppose that $t = O(\log n)$. Subject to Conditions~\ref{cond-regularity-graph} and \ref{cond-regularity-dynamics} below, the following holds whp in $x$ and $\xi$:
\begin{equation}
\label{mainthm3equ}
\Ddyn_{x,\xi}\big(t\,\big) = \prob_{x,\xi}(\tau > t)\,\Dstat_{x,\xi}(t) +o_{\sss\prob}(1).
\end{equation}
\end{theorem}

\noindent
Conditions~\ref{cond-regularity-graph} and \ref{cond-regularity-dynamics} are regularity conditions. The former is rather standard in the literature and ensures that the underlying graph is sparse and that the non-backtracking random walk is well-defined. The latter, representing one of the novelties of this article, ensures that the non-backtracking random walk is well-mixed when it steps along a previously rewired edge and the time at which this happens does not depend on the fine details of its past trajectory. 

The proof of Theorem~\ref{thm:main} is based on a \emph{coupling argument} in which the random walk on the dynamically rewired random graph is coupled to a modified random walk on the static random graph that at certain random times makes uniform jumps. These jumps correspond to the times at which the random walk steps along a previously rewired edge. The coupling must be good enough to beat the errors in the comparison. A key ingredient of the coupling is that the non-backtracking random walk on the configuration model is $\whp$ \emph{self-avoiding} on the scale of the mixing time.

Note that while $\{(X_t,C_t)\}_{t\in\N}$ is Markov, the marginal $\{X_t\}_{t\in\N}$ need not be, even though the stationary distribution of the latter is still the uniform distribution. 


\subsection{Application to specific rewiring mechanisms}

We next consider three choices of random rewiring, referred to as \emph{local-to-global}, \emph{near-to-global} and \emph{global-to-global}, controlled by two parameters: (1) $r_n$, representing the radius of the ball around the current location of the random walk in which edges are allowed to be rewired with an edge that is drawn uniformly at random from the set of all edges; (2) $\alpha_n$, representing the probability that an edge in this ball is rewired per unit of time. By rewiring we mean breaking up two pairs of chosen edges into four half-edges and tying these up at random (for details, see Section \ref{sec:ex}).

At every unit of time a \emph{subset} of the edges is rewired. The rewiring of each edge is always with an edge that is chosen \emph{uniformly at random} from the set of \emph{all} edges. For the subset of edges that is rewired we consider three choices: 
\begin{itemize}
\item\textbf{Local-to-global ($r_n = 1$):}\\
The edge that corresponds to the current position of the random walk has probability $\alpha_n$ to be rewired. 
\item\textbf{Near-to-global ($1<r_n<r_{\mathrm{max}}$):}\\
All the edges in the $r_n$-ball around the current position of the random walk have probability $\alpha_n$ to be rewired, independently of each other.
\item \textbf{Global-to-global ($r_n=r_{\mathrm{max}}$):}\\
All the edges have probability $\alpha_n$  to be rewired, independently of each other. 
\end{itemize}
Here, $r_{\mathrm{max}}$ is the maximal radius (see \eqref{CM:rad} below), provided that the graph is connected (which happens $\whp$ under the conditions that will be stated below). 

Global-to-global rewiring was considered in \cite{AGHH2018} and \cite{AGHH20182}, while local-to-global rewiring was considered in an unpublished chapter of the doctorate thesis \cite{HG2019}. In the present paper, however, we prove results under \emph{weaker assumptions}. For an overview of previous work, see Section~\ref{sec:literature}. Near-to-global rewiring is new and turns out to hold surprises:

\begin{theorem}[Scaling of cross-rewired time]
\label{thm:notmain}
Suppose that $\lim_{n\to\infty}\alpha_n = 0$ and $t=O(\log n)$. Subject to Conditions~\ref{cond-regularity-graph}\ref{cond-regularity-graph-R1} and \ref{cond-regularity-graph-R3} below, the following hold $\whp$ in $x$ and $\xi$:
\begin{enumerate}
\customitem{(A)}\label{thm:notmain-A}
For local-to-global rewiring defined in Section~\ref{sec:locrrw}:
\begin{equation}
\label{ltg}
\prob_{x,\xi}(\tau > t) = o_{\sss\prob}(1) + \mathrm{e}^{-c}, \quad c \in (0,\infty), \qquad t = \lfloor c/\alpha_n\rfloor.
\end{equation}
\customitem{(B)}\label{thm:notmain-B}
For near-to-global rewiring defined in Section~\ref{sec:nearrrw}, subject to Condition~\ref{cond:secmom} below:
\begin{enumerate}
\item
If $\lim_{n\to\infty}\alpha_n r_n^2 = \infty$, then
\begin{equation}
\label{ntg1}
\prob_{x,\xi}(\tau > t) = o_{\sss\prob}(1) + \mathrm{e}^{-c^2/2}, \quad c \in (0,\infty), \qquad t = \lfloor c/\sqrt{\alpha_n}\rfloor.
\end{equation}
\item
If $\lim_{n\to\infty}\alpha_n r_n^2 = \beta \in (0,\infty)$, then
\begin{equation}
\label{ntg2}
\prob_{x,\xi}(\tau > t) = o_{\sss\prob}(1) + \left\{\begin{array}{ll}
\hspace{-.1cm} \mathrm{e}^{-\beta c^2/2}, &c \in (0,1],\\
\hspace{-.1cm} \mathrm{e}^{-\beta(2c-1)/2}, &c \in (1,\infty),
\end{array}
\right.
\quad t = \lfloor c\,r_n\rfloor,
\end{equation}
\item
If $\lim_{n\to\infty}\alpha_n r_n^2 = 0$, then
\begin{equation}
\label{ntg3}
\prob_{x,\xi}(\tau > t) = o_{\sss\prob}(1) + \mathrm{e}^{-c}, \quad c \in (0,\infty), \qquad t = \lfloor c/\alpha_n r_n \rfloor.
\end{equation}
\end{enumerate}
\customitem{(C)}\label{thm:notmain-C}
For global-to-global rewiring defined in Section~\ref{sec:globrrw}: 
\begin{equation}
\label{gtg}
\prob_{x,\xi}(\tau > t) = o_{\sss\prob}(1) + \mathrm{e}^{-c^2/2}, \quad c \in (0,\infty), \qquad t = \lfloor c/\sqrt{\alpha_n}\rfloor.
\end{equation}
\end{enumerate}
\end{theorem}

\noindent
Note that the tail probability $\prob_{x,\xi}(\tau > t)$ exhibits a \emph{trichotomy} for near-to-global rewiring, with an additional crossover at time $t=r_n$ when $\lim_{n\to\infty}\alpha_n r_n^2 = \beta \in (0,\infty)$. 

Condition~\ref{cond:secmom} says that the empirical degree distribution converges to a limit as $n\to\infty$, and so do it first and second moments. It implies that $\whp$ the radius (i.e., the typical distance between vertices) of the random graph is
\begin{equation}
\label{CM:rad}
r_{\mathrm{max}} = [1+o_{\sss\prob}(1)]\,\frac{\log n}{\log \nu},
\end{equation}
where $\nu$ is the size-biased mean of the limiting empirical degree distribution \cite[Theorem 7.1]{vdH2018}, which is assumed to satisfy $\nu \in (1,\infty)$. Thus, for near-to-global rewiring we can only choose 
\begin{equation}
\label{rhomax}
r_n = [\rho+o(1)]\,\log n, \quad \rho \in [0,\rho_\mathrm{max}), \qquad \rho_\mathrm{max} = \frac{1}{\log \nu}.
\end{equation}

Condition~\ref{cond:secmom} is needed for Theorem~\ref{thm:notmain}\ref{thm:notmain-B} only. The fact that it is not needed for Theorem~\ref{thm:notmain}\ref{thm:notmain-C} weakens the conditions in \cite{AGHH2018, AGHH20182}. We expect Theorem \ref{thm:notmain}\ref{thm:notmain-B} to fail without Condition~\ref{cond:secmom}. Namely, when the degree distribution has infinite variance, graph distances are of smaller order than $\log n$, and in fact are of order $\log\log n$ under an appropriate power-law assumption on the empirical degree distribution \cite{CarGarHof19, vdH2018, HHZ2007, HofKom17}. In the latter setting, for $r_n=c \log n$ we expect near-to-global rewiring to behave similarly as global-to-global rewiring. 

In order to exploit Theorem~\ref{thm:main}, we need to also control $\Dstat_{x,\xi}(t)$. For this we use the following result from \cite{BHS2017}, which requires \emph{additional regularity conditions} stronger than Condition~\ref{cond-regularity-graph} (see Appendix~\ref{appB} for further details): 

\begin{theorem}[Scaling of static mixing time]     
\label{thm:stat}
Subject to \eqref{c*def} and Condition~\ref{cond-regularity-graph2}, the following holds whp in $x$ and $\xi$:
\begin{equation}
\Dstat_{x,\xi}(t) = \left\{\begin{array}{lll}
1-o_{\sss\prob}(1), 
&\text{if}\quad t = \lfloor c \log n\rfloor,
&c<c_*,\\
o_{\sss\prob}(1),
&\text{if}\quad t = \lfloor c \log n\rfloor,
&c>c_*,
\end{array}
\right.
\end{equation}
where $c_* \in (0,\infty)$ is the constant defined in \eqref{c*def}.
\end{theorem}

Combining Theorems~\ref{thm:main}--\ref{thm:stat} we end up with the following results:

\begin{corollary}[Scaling of dynamic mixing time for local-to-global rewiring]
\label{locrwmixing}
Consider the local-to-global rewiring defined in Section~\ref{sec:locrrw}. Suppose that $\lim_{n\to\infty}\alpha_n = 0$ and $t=O(\log n)$. Subject to Condition~\ref{cond-regularity-graph}\ref{cond-regularity-graph-R1}, Condition~\ref{cond-regularity-graph2} and  \eqref{c*def}, the following hold whp in $x$ and $\xi$:
\begin{itemize}
\customitem{(1)}\label{locrwmixing-1}
If $\lim_{n\to\infty} \alpha_n\log n = \infty$, then
\begin{equation}
\label{equ:supercritical}
\Ddyn_{x,\xi}\big(\lfloor c/\alpha_n\rfloor\big) = o_{\sss\prob}(1) + \mathrm{e}^{-c}, \quad c \in [0,\infty).
\end{equation}
\customitem{(2)}\label{locrwmixing-2}
If $\lim_{n\to\infty} \alpha_n\log n = \gamma \in (0,\infty)$, then
\begin{equation}
\label{equ:critical3}
\Ddyn_{x,\xi}\big(\lfloor c \log n\rfloor\big) =o_{\sss\prob}(1) + \left\{\begin{array}{ll}
\mathrm{e}^{-\gamma c}, &c \in [0,c_*),\\
0, &c \in (c_*,\infty).
\end{array}
\right.
\end{equation}
\customitem{(3)}\label{locrwmixing-3}
If $\lim_{n\to\infty} \alpha_n\log n = 0$, then
\begin{equation}
\label{equ:subcritical3}
\Ddyn_{x,\xi}\big(\lfloor c \log n\rfloor\big) 
= o_{\sss\prob}(1) + \left\{\begin{array}{ll}
1, &c \in [0,c_*),\\
0, &c \in (c_*,\infty).
\end{array}
\right.
\end{equation}
\end{itemize}
\end{corollary}

\begin{corollary}[Scaling of dynamic mixing time for near-to-global rewiring]
\label{nearrwmixing}
Consider the near-to-global rewiring defined in Section~\ref{sec:nearrrw}. Suppose that $\lim_{n\to\infty} \alpha_n = 0$ and $t=O(\log n)$. Subject to Condition~\ref{cond-regularity-graph}\ref{cond-regularity-graph-R1}, Condition~\ref{cond:secmom}, Condition~\ref{cond-regularity-graph2} and \eqref{c*def}, the following hold whp in $x$ and $\xi$:
\begin{itemize}
\customitem{(1)}\label{nearrwmixing-1}
If $\lim_{n\to\infty} \alpha_n r_n \log n =\infty$ and 
\begin{itemize}
\customitem{(a)}\label{nearrwmixing-1a}
$\lim_{n\to\infty} \alpha_n r_n^2=\infty$, then
\begin{equation}
\label{equ:near:1a}
\Ddyn_{x,\xi}\big(\lfloor c/\sqrt{\alpha_n}\rfloor\big) 
= o_{\sss\prob}(1) + \mathrm{e}^{-c^2/2}, \quad c \in [0,\infty).
\end{equation}
\customitem{(b)}\label{nearrwmixing-1b}
$\lim_{n\to\infty} \alpha_n r_n^2=\beta \in (0,\infty)$, then 
\begin{equation}
\label{equ:near:1b}
\Ddyn_{x,\xi}\big(\lfloor c\,r_n\rfloor\big) 
= o_{\sss\prob}(1) + \left\{\begin{array}{ll}
\hspace{-.1cm} \mathrm{e}^{-\beta c^2/2}, &c \in (0,1],\\
\hspace{-.1cm} \mathrm{e}^{-\beta(2c-1)/2}, &c \in (1,\infty).
\end{array}
\right.
\end{equation}
\customitem{(c)}\label{nearrwmixing-1c}
$\lim_{n\to\infty} \alpha_n r_n^2=0$, then
\begin{equation}
\label{equ:near:1c}
\Ddyn_{x,\xi}\big(\lfloor c/\alpha_nr_n\rfloor\big) 
= o_{\sss\prob}(1) + \mathrm{e}^{-c}, \quad c \in [0,\infty).
\end{equation}
\end{itemize}
\pagebreak
\customitem{(2)}\label{nearrwmixing-2}
If $\lim_{n\to\infty} \alpha_n r_n \log n = \gamma \in (0,\infty)$ and
\begin{itemize}
\customitem{(b)}\label{nearrwmixing-2b}
$\lim_{n\to\infty} \alpha_n r_n^2 = \beta \in (0,\infty)$, then
\begin{equation}
\label{equ:near:2b}
\Ddyn_{x,\xi}\big(\lfloor c \log n\rfloor\big) 
= o_{\sss\prob}(1) + \left\{\begin{array}{ll}
\mathrm{e}^{-(\gamma c)^2/2\beta}, &c \in [0, \beta/\gamma ], \\
\mathrm{e}^{-(2\gamma c-\beta)/2}, &c \in (\beta/\gamma, c_*), \\
0, &c \in (c_*,\infty).
\end{array}
\right.
\end{equation}
\customitem{(c)}\label{nearrwmixing-2c}
$\lim_{n\to\infty} \alpha_n r_n^2 = 0$, then
\begin{equation}
\label{equ:near:2c}
\Ddyn_{x,\xi}\big(\lfloor c \log n\rfloor\big) 
= o_{\sss\prob}(1) + \left\{\begin{array}{ll}
\mathrm{e}^{-\gamma c}, &c \in  (0, c_*),\\ 
0, &c \in (c_*,\infty).
\end{array}
\right.
\end{equation}
\end{itemize}
\customitem{(3)}\label{nearrwmixing-3} 
If $\lim_{n\to\infty}\alpha_n r_n \log n = 0$ and 
\begin{itemize}
\customitem{(c)}\label{nearrwmixing-3c}
$\lim_{n\to\infty} \alpha_n r_n^2 = 0$, then
\begin{equation}
\label{equ:near:3c}
\Ddyn_{x,\xi}\big(\lfloor c \log n\rfloor\big) 
= o_{\sss\prob}(1) + \left\{\begin{array}{ll}
1, &c \in [0,c_*),\\
0, &c \in (c_*,\infty).
\end{array}
\right.
\end{equation}
\end{itemize}
\end{itemize}
\end{corollary}

\begin{corollary}[Scaling of dynamic mixing time for global-to-global rewiring]
\label{globalrwmixing}
Consider the global-to-global rewiring defined in Section~\ref{sec:globrrw}. Suppose that $\lim_{n\to\infty}\alpha_n = 0$ and $t=O(\log n)$. Subject to Condition~\ref{cond-regularity-graph}\ref{cond-regularity-graph-R1}, Condition~\ref{cond-regularity-graph2} and \eqref{c*def}, the following hold whp in $x$ and $\xi$:
\begin{itemize}
\customitem{(1)}\label{globalrwmixing-1}
If $\lim_{n\to\infty}\alpha_n ( \log n )^2 = \infty$, then
\begin{equation}
\label{equ:globalsupercritical}
\Ddyn_{x,\xi}\big(\lfloor c/\sqrt{\alpha_n}\rfloor\big) 
= o_{\sss\prob}(1) + \mathrm{e}^{-c^2/2}, \quad c \in [0,\infty).
\end{equation}
\customitem{(2)}\label{globalrwmixing-2}
If $\lim_{n\to\infty}\alpha_n ( \log n )^2 = \gamma \in (0,\infty)$, then
\begin{equation}
\label{equ:globalcritical3}
\Ddyn_{x,\xi}\big(\lfloor c \log n\rfloor\big) 
= o_{\sss\prob}(1) + \left\{\begin{array}{ll}
\mathrm{e}^{-\gamma c^2/2}, &c \in [0,c_*),\\
0, &c \in (c_*,\infty).
\end{array}
\right.
\end{equation}
\customitem{(3)}\label{globalrwmixing-3}
If $\lim_{n\to\infty} \alpha_n (\log n)^2= 0$, then
\begin{equation}
\label{equ:globalsubcritical3}
\Ddyn_{x,\xi}\big(\lfloor c \log n\rfloor\big) 
= o_{\sss\prob}(1) + \left\{\begin{array}{ll}
1, &c \in [0,c_*),\\
0, &c \in (c_*,\infty).
\end{array}
\right.
\end{equation}
\end{itemize}
\end{corollary}

Note that the dynamic mixing time exhibits a \emph{trichotomy} that distinguishes between \emph{fast} dynamics (regime (1)), \emph{moderate} dynamics (regime (2)) and \emph{slow} dynamics (regime (3)). There is \emph{no} cut-off for fast dynamics, \emph{one-sided} cut-off (at $c=c_*$) for moderate dynamics, and \emph{two-sided} cut-off (at $c=c_*$) for slow dynamics. For near-to-global rewiring there are several subregimes (regimes (2)(a), (3)(a) and (3)(b) are not relevant). See Figs.~\ref{fig:tvdisttrichotomy-ltg}--\ref{fig:tvdisttrichotomy-ntg} for the various scaling shapes (where the indices $x$ and $\xi$ are suppressed).

\begin{figure}[p]
\centering
\begin{minipage}{.48\textwidth}
\centering
\begin{subfigure}{\linewidth}
\centering
\includegraphics[width=\linewidth]{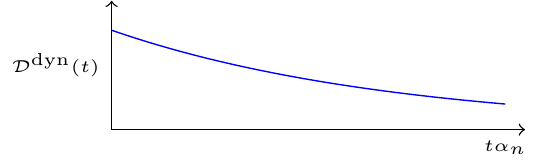}
\subcaption*{Regime \ref{locrwmixing-1}: \eqref{equ:supercritical}.}
\label{fig:ltg:supercrit}
\end{subfigure}
\begin{subfigure}{\linewidth}
\centering
\includegraphics[width=\linewidth]{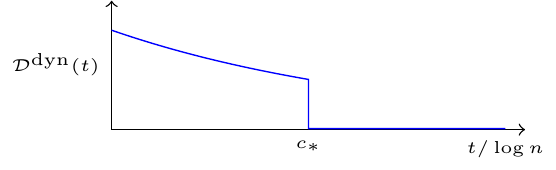}
\subcaption*{Regime \ref{locrwmixing-2}: \eqref{equ:critical3}.}
\label{fig:ltg:crit}
\end{subfigure}
\begin{subfigure}{\linewidth}
\centering
\includegraphics[width=\linewidth]{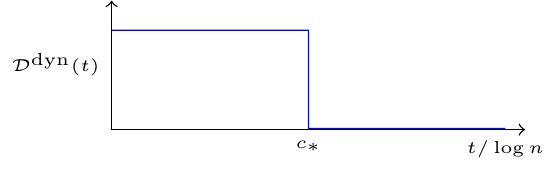}
\subcaption*{Regime \ref{locrwmixing-3}: \eqref{equ:subcritical3}.}
\label{fig:ltg:subcrit}
\end{subfigure}
\caption{Plot of $\Ddyn(t)$ for local-to-global rewiring (Corollary~\ref{locrwmixing}).}
\label{fig:tvdisttrichotomy-ltg}
\end{minipage}
\hfill\vline\hfill
\begin{minipage}{.48\textwidth}
\centering
\begin{subfigure}{\linewidth}
\centering
\includegraphics[width=\linewidth]{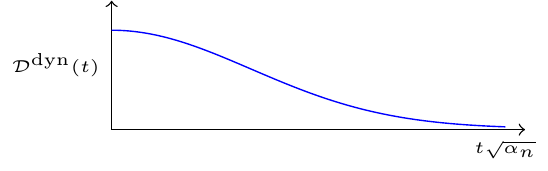}
\subcaption*{Regime \ref{globalrwmixing-1}: \eqref{equ:globalsupercritical}.}
\label{fig:gtg:supercrit}
\end{subfigure}
\begin{subfigure}{\linewidth}
\centering
\includegraphics[width=\linewidth]{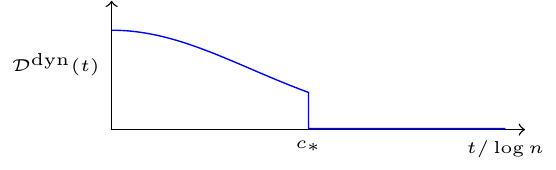}
\subcaption*{Regime \ref{globalrwmixing-2}: \eqref{equ:globalcritical3}.}
\label{fig:gtg:crit}
\end{subfigure}
\begin{subfigure}{\linewidth}
\centering
\includegraphics[width=\linewidth]{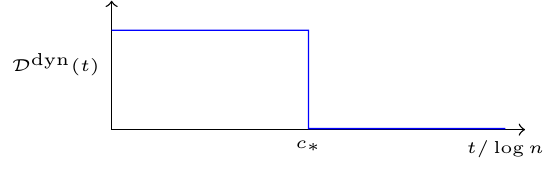}
\subcaption*{Regime \ref{globalrwmixing-3}: \eqref{equ:globalsubcritical3}.}
\label{fig:gtg:subcrit}
\end{subfigure}
\caption{Plot of $\Ddyn(t)$ for global-to-global rewiring (Corollary~\ref{globalrwmixing}).}
\label{fig:tvdisttrichotomy-gtg}
\end{minipage}
\end{figure}

\begin{figure}[p]
\centering
\begin{subfigure}{.5\textwidth}
\centering
\includegraphics[width=\textwidth]{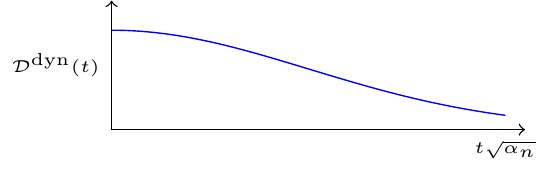}
\subcaption*{Regime 1\ref{nearrwmixing-1a}: \eqref{equ:near:1a}.}
\label{fig:ntl:1a}
\end{subfigure}%
\begin{subfigure}{.5\textwidth}
\centering
\includegraphics[width=\textwidth]{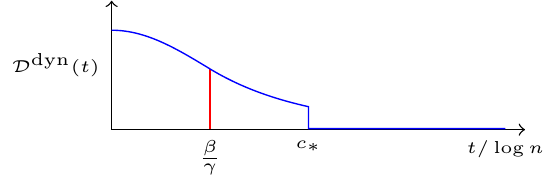}
\subcaption*{Regime 2\ref{nearrwmixing-2b}: \eqref{equ:near:2b}.}
\label{fig:ntl:2b}
\end{subfigure}
\begin{subfigure}{.5\textwidth}
\centering
\includegraphics[width=\textwidth]{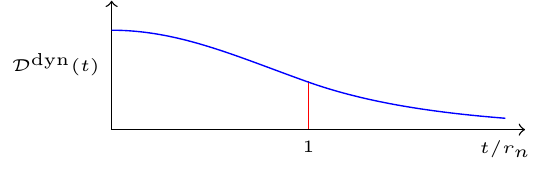}
\subcaption*{Regime 1\ref{nearrwmixing-1b}: \eqref{equ:near:1b}.}
\label{fig:ntl:1b}
\end{subfigure}%
\begin{subfigure}{.5\textwidth}
\centering
\includegraphics[width=\textwidth]{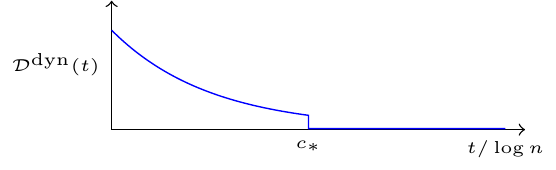}
\subcaption*{Regime 2\ref{nearrwmixing-2c}: \eqref{equ:near:2c}.}
\label{fig:ntl:2c}
\end{subfigure}
\begin{subfigure}{.5\textwidth}
\centering
\includegraphics[width=\textwidth]{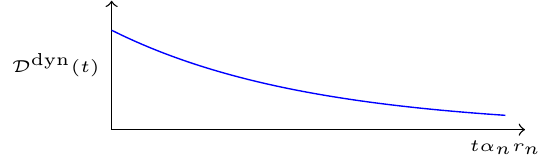}
\subcaption*{Regime 1\ref{nearrwmixing-1c}: \eqref{equ:near:1c}.}
\label{fig:ntl:1c}
\end{subfigure}%
\begin{subfigure}{.5\textwidth}
\centering
\includegraphics[width=\textwidth]{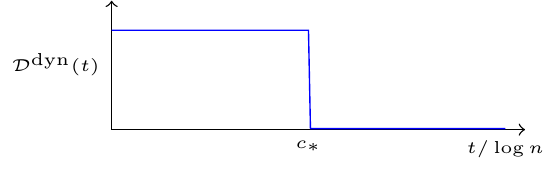}
\subcaption*{Regime 3\ref{nearrwmixing-3c}: \eqref{equ:near:3c}.}
\label{fig:ntl:3c}
\end{subfigure}
\caption{Plot of $\Ddyn(t)$ for near-to-global rewiring (Corollary~\ref{nearrwmixing}). The red lines indicate a crossover in the shape of the curve.}
\label{fig:tvdisttrichotomy-ntg}
\end{figure}

\begin{remark}[Role of Condition~\ref{cond-regularity-dynamics}]
Corollaries~\ref{locrwmixing}--\ref{globalrwmixing} do not mention Condition~\ref{cond-regularity-dynamics} explicitly, even though this is needed for Theorem~\ref{thm:main}. The reason is that the three rewiring mechanisms under consideration satisfy Condition~\ref{cond-regularity-dynamics}, as shown in Section~\ref{sec:ex}.
\end{remark}


\subsection{Discussion}
\label{sec:disc}
\newcounter{counterdisc}

{\bf \stepcounter{counterdisc}\thecounterdisc.}
Each of the three choices of rewiring shows a \emph{trichotomy} between fast dynamics ($\alpha_nr_n \gg 1/\log n)$, moderate dynamics ($\alpha_nr_n \asymp 1/\log n)$ and slow dynamics ($\alpha_nr_n \ll 1/\log n)$, with $r_n=1$ for local-to-global rewiring, $1 < r_n < r_{\mathrm{max}}$ for near-to-global rewiring and $r_n=r_{\mathrm{max}}$ for global-to-global rewiring. For fast dynamics the mixing time is of smaller order than $\log n$, which is the mixing time on the static random graph, and so \emph{speed-up} occurs. For moderate and slow dynamics the mixing time is of order $\log n$, and so \emph{no speed-up} occurs. The one-sided cut-off for moderate dynamics shows that there is a \emph{competition} between static and dynamic. For fast dynamics only Conditions~\ref{cond-regularity-graph} and \ref{cond-regularity-dynamics} are needed, while for moderate and slow dynamics \eqref{c*def} and Condition~\ref{cond-regularity-graph2} are needed as well. For fast dynamics the scaling does \emph{not} depend on the choice of degrees, subject to the mild regularity imposed by Condition~\ref{cond-regularity-graph}. On other hand, for moderate and slow dynamics it does, because the constant $c_*$ equals the limit as $n\to\infty$ of the empirical average of the logarithm of the degrees of the half-edges.   

\medskip\noindent
{\bf \stepcounter{counterdisc}\thecounterdisc.}
Whereas for local-to-global and global-to-global rewiring the trichotomy controls the scaling, for near-to-global rewiring several \emph{subregimes} show up. In particular, crossovers in the mixing time occur at critical values of the scaling parameter $c$ (see \eqref{equ:near:1b} and \eqref{equ:near:2b}). These arise from a crossover in the cross-rewired time that appears as soon as $t  \asymp r_n$ (see \eqref{ntg2}). What happens is that all edges on the $r_n$-future of the path can be rewired before the random walk reaches them, but only until time $t-r_n$: for any time $s \in (t-r_n,t]$ only $t-s$ edges are left on the future path until time $t$. The extra condition in Condition~\ref{cond:secmom} ensures that $\whp$ the $r_n$-balls carried around by the random walk do not overlap significantly, i.e., short-cuts of length $\leq r_n$ are negligible until time $t=O(\log n)$.

\medskip\noindent        
{\bf \stepcounter{counterdisc}\thecounterdisc.}
Regime (2b) for near-to-global rewiring corresponds to $\rho = \beta/\gamma$. Subject to Condition~\ref{cond:secmom} we have
\begin{equation}
\frac{1}{c_*} =  \sum_{m\in\N} p^\star(m) \log m,
\end{equation} 
with
\begin{equation}
p^\star(m) = \frac{1}{N} (m+1) p(m+1), \quad m \in \N_0, \qquad N = \sum_{m\in\N} m p(m), 
\end{equation}
where $p(m)=\lim_{n\to\infty} p_n(m)$ with $p_n=\frac{1}{n} \sum_{v \in V} \delta_{\deg(v)}$ the empirical degree distribution (see \eqref{degdef3}, Theorem~\ref{thm:scalstat} and \eqref{c*def}; Condition~\ref{cond-regularity-graph2}, which is needed for Theorem~\ref{thm:scalstat}, implies that $p(1)=p(2)=0$). By Jensen's inequality, 
\begin{equation}
\sum_{m\in\N} p^\star(m) \log m \leq \log\left(\sum_{m\in\N} p^\star(m)\,m\right)=\log{\nu}.
\end{equation}
Consequently, $c_* \geq \rho_\mathrm{max}$ (see \eqref{rhomax}, \eqref{nudef} and \eqref{c*def}), with equality if and only if $p$ is a point mass. Thus, the cut-off threshold $c_*$ exceeds the maximal value of the radius, as shown in Fig.~\ref{fig:tvdisttrichotomy-ntg}.     
 
\medskip\noindent
{\bf \stepcounter{counterdisc}\thecounterdisc.} 
The coupling of the random walk on the dynamically rewired random graph to the modified random walk is implicit in the proof of the main theorem in \cite{AGHH20182}. There the main idea was that the path probabilities for the two random walks coincide for self-avoiding paths, and it was shown that the two random walks are with high probability self-avoiding. The crucial observation was that, on a typical configuration drawn according to the configuration model, the random walks are self-avoiding with high probability. The particular form of Condition~\ref{cond-regularity-dynamics} was motivated by this observation, and suggests that the same results may hold when the initial graph is drawn according to some other distribution, on which non-backtracking random walks are typically self-avoiding.

\medskip\noindent
{\bf \stepcounter{counterdisc}\thecounterdisc.}
The graph regularity conditions in Condition~\ref{cond-regularity-dynamics} are mild, but can be violated. Consider for example a modification of the local-to-global rewiring in which the probability $\alpha_n$ of the half-edge $X_{t-1}$ being rewired at time $t$ depends on a specific choice of $X_{t-1}$, e.g.\ $\alpha_n(X_{t-1}) = 1/\deg_H(X_{t-1})$. This would lead to a violation of Condition~\ref{cond-regularity-dynamics}\ref{cond-regularity-dynamics-D1}. Condition~\ref{cond-regularity-dynamics}\ref{cond-regularity-dynamics-D2} can be violated by a graph rewiring mechanism that at each time gives preferential treatment to some half-edges. For example, fix a set of half-edges $F$ with $|F|\ll n$, and define a \enquote{local-to-F} graph dynamics where the edge that might get rewired at time $t$ with $\{X_{t-1}, C_{t-1}(X_{t-1})\}$ is chosen from the set of edges generated by the configuration $C_{t-1}$ such that each edge contains at least half-edge from $F$. This obviously results in a violation of the Condition~\ref{cond-regularity-dynamics}\ref{cond-regularity-dynamics-D2}.

\medskip\noindent
{\bf \stepcounter{counterdisc}\thecounterdisc.}
The scaling regimes considered in Theorem~\ref{thm:notmain}, Corollaries~\ref{locrwmixing}--\ref{globalrwmixing} and Figures~\ref{fig:tvdisttrichotomy-ltg}--\ref{fig:tvdisttrichotomy-ntg} are chosen so as to end up with \emph{non-trivial scaling profiles}. Apart from conditions on $t$ in terms of $\alpha_n$ and $r_n$, there is also the implicit condition that $t = O(\log n)$. However, since the probability that $\tau>t$ is monotone decreasing in both $\alpha_n$ and $r_n$,  we get trivial scaling profiles outside these regimes. 


\subsection{Previous work}
\label{sec:literature}

The past decade has witnessed much activity towards understanding processes -- both random and deterministic -- on dynamic networks \cite{MS2018,C2011,FNRdST2012,CST2015,BGKM2016,GSS2014,APR2016,KO2011,SRP2015,WPRS2019}. Research is motivated not only by mathematical interest, but also by numerous applications in computer science and data science. One of the emerging efforts is concerned with the study of mixing times of random walks on dynamic networks, and how they compare with those of random walks on static networks. The present paper fits within this line of research.

In \cite{AGHH2018} we introduced a version of a \emph{dynamic configuration model} in which a fraction of the edges gets rewired at each step of the random walk according to a global-to-global rewiring mechanism. We obtained an expression for the mixing time of a \emph{non-backtracking} random walk under conditions that guarantee a locally tree-like structure of the graph and fast dynamics. In \cite{AGHH20182} we extended our results to moderate and slow dynamics. In particular, we obtained a \emph{trichotomy} for the mixing time of non-backtracking random walks, of the type as stated in Corollary~\ref{globalrwmixing}. In the current paper, however, we achieve this trichotomy under \emph{weaker assumptions}.

Trichotomies were also found in subsequent work. The closest to our setting is \cite{CQ2019-2}, where the authors consider a dynamic \emph{directed} version of the configuration model. Contrary to our setting, for the directed graph the rewiring no longer preserves the stationary measure, and the analysis in\cite{CQ2019-2} is restricted to a rewiring mechanism in which all the edges are freshly resampled at each step of the random walk. Two trichotomies are derived for the worst-case total variation distance, respectively, for the joint Markov process given by the graph and the random walk and for the non-Markov process given by the random walk marginal. Trichotomies can also emerge in the presence of other random mechanisms that do not directly change the graph. This is well illustrated in \cite{CQ2019,VS2019,WPRS2019}, where crossovers were established for random walks on random graphs with various PageRank-like transitions. Results are analogous to Theorem~\ref{thm:main}, with the role of the randomized stopping time $\tau$ replaced by the first time the walk gets \enquote{teleported} by a PageRank-like transition.

Mixing studies for random walks on dynamic random graphs started with \cite{PSS2015}, which considered random walks on \emph{dynamic percolation} clusters on a $d$-dimensional \emph{discrete torus}, i.e., a stochastic version of percolation where edges appear and disappear independently at a given rate. In \cite{PSS2015} and subsequent works \cite{PSS2017,HS2019}, mixing times were identified for several parameter regimes controlling the rates of the random walk and the random graph dynamics. Similar results were obtained for dynamic percolation on the \emph{complete graph} \cite{ST2018,SZ2019}. One of the main difficulties with the dynamic percolation setting is that the stationary distribution of the random walk changes over time, which explains why results tend to be restricted to specific parameter regimes.

Some further advances were achieved in \cite{SZ2019,CKL2018}, where general bounds on mixing times, and other quantities such as hitting, cover and return times, were derived for certain classes of evolving graphs under proper \emph{expansion assumptions}. Typically, random walk mixing on a dynamic graph is \emph{faster} than on a static graph, although \cite{CKL2018} contains some (artificial) examples where the dynamics makes the mixing slower. Speed-up of mixing times for general Markov chains was recently analysed in \cite{CD2020}, which also contains an overview of related results. 

Unlike for dynamic graphs, mixing times of random walks on \emph{static} random graphs form a well-established subject. For the present paper it is important to note the work in \cite{LS2010,BLPS2018,BHS2017}, where (two-sided) cut-offs on time scale $\log n$ were established for both simple and non-backtracking random walks on a fairly general class of sparse \emph{undirected} random graphs with \emph{good expansion properties}. More recently, similar results were obtained for static random graphs with \emph{directed} edges \cite{BCS2018,BCS2019} or with a \emph{community structure} \cite{BH2018}.


\subsection{Outline}

The remainder of this paper is organised as follows. In Section~\ref{sec:RWDRG} we define the random walk and the random graph dynamics. In Section~\ref{sec:mainthm} we prove Theorem~\ref{thm:main}. In Section~\ref{sec:ex} we prove Theorem~\ref{thm:notmain}. In Appendix~\ref{appA} we show that the joint Markov chain of random walk and dynamically rewired random graph is irreducible, aperiodic and doubly-stochastic. In Appendix~\ref{appB} we recall the precise form of Theorem~\ref{thm:scalstat}. In Appendix~\ref{appC} we identify the general form of the transition matrix for rewirings and prove that the stationary distribution for the class of \enquote{anything-to-global} rewirings is the uniform distribution on $H$.


\section{Random graph dynamics and random walk}
\label{sec:RWDRG}

In this section we set up the model. In Section~\ref{ss:rew} we give a general description of the rewiring mechanism for the random graph (specific choices will be considered in Section~\ref{sec:ex}). In Section~\ref{ss:rw} we define the non-backtracking random walk. In Section~\ref{ss:jointMC} we define the joint process of random graph and random walk.
 

\subsection{Random graph dynamics}
\label{ss:rew}

We consider a general class of graph dynamics in which some edges are randomly rewired at each unit of time according to a prescribed rule. First a subset of edges to be rewired is chosen randomly, then these edges are broken into half-edges, and afterwards the resulting half-edges are paired randomly according to a prescribed distribution. The set of half-edges involved in the rewiring at time $t\in\N$ is denoted by $R_t$.

Suppose that $X_{t-1} = x$ and $C_{t-1} = \xi$. Then, at time $t$, the above dynamics gives rise to a distribution $Q_x(\xi,\cdot)$ on $\Conf_H$. In \cite{AGHH2018}, \cite{AGHH20182} a specific choice of dynamics was considered in which $Q_x(\xi,\cdot)$ did not actually depend on $x$. In such a situation, the configuration component forms a Markov chain itself.


\subsection{Random walk} 
\label{ss:rw}

We consider a \emph{non-backtracking random walk} on a dynamic random graph in which some edges are rewired at each step. By non-backtracking we mean that the random walk cannot traverse the same edge \emph{twice in a row}. Since in our model the underlying graph is dynamic and the edges change over time, the random walk is more conveniently defined as a random walk on the set of half-edges $H$. Recall that at time $t\in\N$ we update the configuration to $C_t = \xi$ and only then let the random walk make a move. Then the random walk moves according to the transition probabilities
\begin{align}
\label{NBT3}
P_\xi(x,y) \coloneqq
\begin{cases}
\frac{1}{\Hdeg(y)} & \text{if }\xi(x)\sim y\text{ and }\xi(x)\neq y, \\
0 & \text{otherwise}.
\end{cases}
\end{align}
More descriptively, when the random walk is on a half-edge $x$ and the graph is in configuration $\xi$, the random walk moves to one of the siblings of the half-edge that the current half-edge $x$ is paired with, chosen uniformly at random (see Fig.~\ref{fig:nbrw_move3}). The transition probabilities are symmetric with respect to the pairing given by $\xi$, i.e., $P_\xi(x,y) = P_\xi(\xi(y),\xi(x))$. In particular, the transition matrix is doubly stochastic, and so the uniform distribution on $H$, denoted by $U_H$, is the stationary distribution for the random walk process:

 \begin{figure}[htbp]
 \centering
 \includegraphics[width=0.20\textwidth]{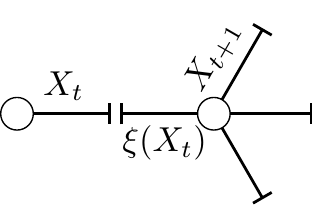}
 \caption{\small The random walk moves from half-edge $X_t$ to half-edge $X_{t+1}$, one of 
 the siblings of the half-edge $\xi(X_t)$ that $X_t$ is paired to.}
 \label{fig:nbrw_move3}
 \end{figure}

\subsection{Joint process} 
\label{ss:jointMC}

The law of the joint Markov chain $(X_t,C_t)_{t\in\N}$, starting from initial half-edge $x\equiv X_0$ and initial configuration $\xi \equiv C_0$, is given by the conditional probabilities
\begin{equation}
\prob_{x,\xi}(X_t = z,C_t = \zeta\mid X_{t-1} = y, C_{t-1} = \eta) 
=Q_y(\eta,\zeta)P_\zeta(y,z),\quad t\in\N,
\end{equation}
with
\begin{equation}
\prob_{x,\xi}(X_0 = x,C_0 = \xi) = 1,
\end{equation}
where the transition probabilities $Q_y(\cdot,\cdot)$ remain to be chosen. While the joint process is Markov, the marginal processes $X = (X_t)_{t\in\N}$ and $C = (C_t)_{t\in\N}$ need \emph{not} be Markov. Consequently, the total variation distance $\|\prob_{x,\xi}(X_t\in\cdot) - U_H(\cdot)\|_{\TV}$ is not guaranteed to be decreasing in $t$, even when it converges to 0.

We emphasise that at each time step the graph evolution happens first and only then the random walk makes a move.

Furthermore, note that when the graph dynamics does not depend on the random walk, i.e., $Q_x(\cdot,\cdot) = Q_y(\cdot,\cdot)$ for all $x,y\in H$, the uniform distribution $U_H$ is the stationary distribution for the random walk, i.e., for all $\xi\in\Conf_H$ and $t\in\N$,
\begin{equation}
\sum_{x\in H}\frac{1}{\sizeH}\prob_{x,\xi}(X_t\in\cdot) = U_H(\cdot).
\end{equation}
This can easily be seen by noting that the random walk conditioned on a realisation of the graph dynamics is a time-inhomogeneous Markov chain for which $U_H$ is the stationary distribution. 


\section{Proof of the main theorem}
\label{sec:mainthm}

In this section we build up the apparatus that is required to prove Theorem~\ref{thm:main}. In Section~\ref{ss:regconds} we formulate the regularity conditions for the graph and its evolution. In Section~\ref{sec:modrw} we introduce the modified random walk, which lives on the static random graph. In Section~\ref{sec:coupling} we propose a coupling of the modified random walk and the dynamically rewired random walk. In Section~\ref{sec:fail} we analyse the errors in the coupling. In Section~\ref{sec:proofmainthm3} we use the coupling to prove Theorem~\ref{thm:main}.


\subsection{Regularity conditions}
\label{ss:regconds}

In the formulation of Theorem~\ref{thm:main} we refer to certain regularity conditions, which we lay out next. The first set of conditions concerns the degrees of the graph:

\begin{condition}[Regularity of degrees]\label{cond-regularity-graph}$\mbox{}$
\begin{itemize} 
\customitem{(R1)}\label{cond-regularity-graph-R1}
$\sizeH =\Theta(n)$ as $n\to\infty$.
\customitem{(R2)}\label{cond-regularity-graph-R2}
$\max \limits_{v\in V}\deg(v)\eqqcolon\dmax = o(n/(\log n)^2)$ as $n\to\infty$.
\customitem{(R3)}\label{cond-regularity-graph-R3} 
$\deg(v)\geq 2$ for all $v\in V$.
\end{itemize}
\end{condition}

\noindent
Condition~\ref{cond-regularity-graph}\ref{cond-regularity-graph-R1} ensures that the graph is sparse, and together with Condition~\ref{cond-regularity-graph}\ref{cond-regularity-graph-R2} guarantees that the paths of the random walk are with high probability self-avoiding on relevant time scales (see Lemma \ref{mainlemma} below). Condition~\ref{cond-regularity-graph}\ref{cond-regularity-graph-R3} is a consistency condition ensuring that the non-backtracking random walk is well-defined.

As stated in the introduction, Condition~\ref{cond-regularity-graph} is standard in the literature, unlike the forthcoming Condition~\ref{cond-regularity-dynamics}. To state this new condition, we require further notation:

\begin{definition}[Path-tracking sequences]
\label{def:pathseq}
For $s,t \in\N$ with $s \leq t$, define $[s,t] \coloneqq \{ s,\dots,t\}$ and $[t] \coloneqq [1,t] = \{1,\dots,t\}$. For $r\in\N$, $t_1,\ldots,t_r\in\N$ with $t_1<\dots<t_r\leq t-1$, introduce a set of times 
\begin{equation}
\label{eq:def:T}
T \coloneqq \{ t_1,\dots,t_r \},
\end{equation}
and sequences of half-edges
\begin{equation}
\begin{split}
x_{[0,t-1]} \coloneqq \left( x_0, \dots, x_{t-1}\right),&
\quad \bar{x}_{[0,t-1]} \coloneqq \left(\bar{x}_0, \dots, \bar{x}_{t-1}\right),\\
\hat{x}_{[r]} \coloneqq \left(\hat{x}_1, \dots, \hat{x}_r\right),&
\quad \tilde{x}_{[r]} \coloneqq \left(\tilde{x}_1, \dots, \tilde{x}_r\right).
\end{split}
\end{equation}
\end{definition}

\vspace{-0.3cm}
\begin{figure}[htbp]
\centering
\includegraphics[scale=0.8]{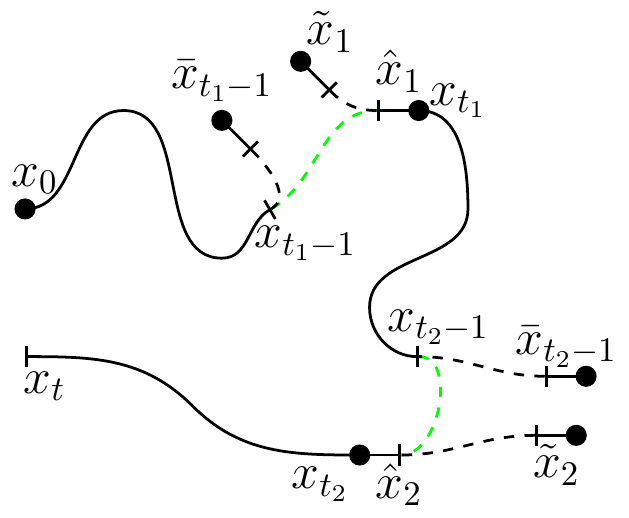}
\caption{\small Illustration of the event $\eventDSA$ and the role of the sequences in Definition~\ref{def:pathseq}. Dashed black lines show the pairing in the initial configuration, while dashed green lines show the pairings that are seen by the random walk.}
\label{fig:DS}
\end{figure}

\begin{definition}[Dynamic self-avoidance]
\label{def:DSA} 
The sequences $x_{[0,t-1]}, \bar{x}_{[0,t-1]}, \tilde{x}_{[r]}$ are called \emph{dynamically self-avoiding with respect to} $T$ if the sequences of vertices 
\begin{equation}
\left(v(x_0),\dots,v(x_{t-1})\right), \quad \left(v(\bar{x}_{t_1-1}),\dots,v(\bar{x}_{t_r-1})\right), 
\quad \left(v(\tilde{x}_1),\dots,\allowbreak v(\tilde{x}_r)\right),
\end{equation} 
are all distinct. Let $\eventDSA$ be the event that (see Fig.~\ref{fig:DS}):
\begin{itemize}
\item $x_{[0,t-1]}, \bar{x}_{[0,t-1]}, \tilde{x}_{[r]}$ are dynamically self-avoiding with respect to $T$.
\item $I_s = 1$ for $s\in T$ and $I_s = 0$ for $s\in[t-1]\setminus T$.
\item $C_0(x_s) = \bar{x}_s$ for $s = 0,\dots,t-1$, where $C_0$ is the configuration at time $t=0$.
\item $C_{t_i}(x_{t_i-1}) = \hat{x}_i$ for $i = 1,\dots,r$. 
\item $C_0(\hat{x}_i) = \tilde{x}_i$ for $i = 1,\dots,r$.
\item $X_s = x_s$ for $s = 0,\dots,t-1$.
\end{itemize}
When the event $\eventDSA$ occurs, we say that the random walk on the dynamic random graph has a \emph{dynamically self-avoiding history up to time} $t$. We call $x_{[0, t-1]}$ \emph{good} when $\Hdeg(x_i) \leq \maxHdeg$ for $0 \leq i \leq t-1$ and some $\epsilon>0$ fixed, i.e., for all half-edges in the sequence $x_{[0, t-1]}$ the degree is $O(\maxHdeg)$. A sequence $x_{[0, t-1]}$ that is not good is called bad.
\end{definition}

\begin{remark}[Interpretation of sequences]
\label{rem:sequences}
 On the event $\eventDSA$, the sequences in Definition~\ref{def:pathseq} have the following interpretation:
\begin{itemize}
\item $T \coloneqq \{ t_1,\dots,t_r  \} \subset [t-1]$:\\ 
times when the random walk steps along a previously rewired edge, 
\item $x_{[0,t-1]} \coloneqq \left(x_0, \dots, x_{t-1}\right)$:\\ 
half-edges the random walk visits up to time $t-1$.
\item $\bar{x}_{[0,t-1]} \coloneqq \left(\bar{x}_0, \dots, \bar{x}_{t-1}\right)$:\\ 
half-edges $\bar{x}_i$ that are paired with half-edges $x_i \in x_{[0,t-1]}$ in the initial configuration.
\item $\hat{x}_{[r]} \coloneqq \left(\hat{x}_1, \dots, \hat{x}_r\right)$:\\ 
half-edges that are paired with $x_{i-1}$ at times $ i \in T$.
\item $\tilde{x}_{[r]} \coloneqq \left(\tilde{x}_1, \dots, \tilde{x}_r\right)$:\\ 
half-edges that are paired with $\hat{x}_1,\dots,\hat{x}_r$ in the initial configuration.
\end{itemize}
\end{remark}

With these definitions in hand, we can now state the conditions on the random graph dynamics:

\begin{condition}[Regularity of graph dynamics]\label{cond-regularity-dynamics}
Recall that $I_t$ is the indicator of the event that a random walk steps over a rewired edge at time~$t$ (see Definition~\ref{def:stopping}). For all $t = O(\log n)$ and all $T = \{t_1,\dots,t_r\}\subset[t-1]$ the following conditions hold (note that $I_t$ is random given $(I_s)_{0 < s < t}$): 
\begin{itemize}
\customitem{(D1)}\label{cond-regularity-dynamics-D1}
For all $\allx$ and $\ally$ that describe dynamically self-avoiding histories with respect to $T$, 
\begin{equation}
\begin{aligned}
&\big|\prob(I_t = 1\mid \eventDSA)\\
&\qquad -\prob(I_t = 1\mid \eventDSAy)\big| = o(\tfrac{1}{\log n}),
\end{aligned}
\end{equation}
where the bound is \emph{uniform} in the histories.
\customitem{(D2)}\label{cond-regularity-dynamics-D2}
For all $\allx$ that describe dynamically self-avoiding histories with respect to $T$,
\begin{equation}
\|\prob(C_t(x_{t-1})\in\cdot\mid\eventDSA\cap\{I_t = 1\}) - U_H(\cdot)\|_{\TV} = o(\tfrac{1}{\log n}),
\end{equation} 
where the bound is \emph{uniform} in the histories.
\end{itemize}
In case \ref{cond-regularity-dynamics-D1} and \ref{cond-regularity-dynamics-D2} cannot be verified for \emph{all} sets of sequences that describe a dynamically self-avoiding history with respect to $T$, the following suffices:
\begin{itemize}
\customitem{(D3)}\label{cond-regularity-dynamics-D3}
If \ref{cond-regularity-dynamics-D1} and \ref{cond-regularity-dynamics-D2} hold for $\allx$ and $\ally$ that describe dynamically self-avoiding histories with respect to $T$ for which $x_{[0, t-1]}$ and $y_{[0, t-1]}$ are \emph{good}, but not necessarily hold for those histories for which $x_{[0, t-1]}$ and $y_{[0, t-1]}$ are bad, then the path $X_{[0, t]}$ of the random walk on the dynamic random graph traced until time $t$ satisfies:
\begin{equation}
\prob\left( X_{[0, t-1]}\text{ is bad}\,\right) = o(\tfrac{1}{\log n}).
\end{equation}
\end{itemize}
\end{condition}

\noindent
Part \ref{cond-regularity-dynamics-D1} states that the times at which the random walk steps over a rewired edge are almost independent of the fine details of the random walk, provided it has a good dynamically self-avoiding history. Part \ref{cond-regularity-dynamics-D2} states that a random walk with a good dynamically self-avoiding history is close to being mixed right after it steps over a rewired edge. The error terms of order $o(1/\log n)$ are chosen such that we can carry out the estimates in Lemma~\ref{mainlemma} below. Part~\ref{cond-regularity-dynamics-D3} ensures that good dynamically self-avoiding histories are typical.

To identify the scaling of the mixing time for near-to-global rewiring in Corollary~\ref{nearrwmixing}, we need an extra regularity condition:

\begin{condition}[Regularity of degree distribution]
\label{cond:secmom}
Let $p_n := \frac{1}{n} \sum_{v \in V} \delta_{\deg(v)}$ denote the empirical degree distribution. We require that:
\begin{itemize}
\customitem{(R1*)}\label{cond:secmom-R1}
$\lim_{n\to\infty} p_n = p$, pointwise for some probability distribution $p$ on $\N$.
\customitem{(R2*)}\label{cond:secmom-R2}
$\lim_{n\to\infty} \sum_{m\in\N} mp_n(m) = \sum_{m\in\N} mp(m) < \infty$. 
\customitem{(R3*)}\label{cond:secmom-R3}
$\lim_{n\to\infty} \sum_{m\in\N} m^2p_n(m) = \sum_{m\in\N} m^2p(m) < \infty$.
\end{itemize}

\noindent
The size-biased mean minus one of $p$ is
\begin{equation}
\label{nudef}
\nu = \frac{\sum_{m\in\N} m(m-1)p(m)}{\sum_{m\in\N} mp(m)}
\end{equation} 
and is assumed to satisfy $\nu>1$. 
\end{condition}

\noindent
We may interpret $\nu$ as the \emph{average forward degree} of a uniformly chosen half-edge, which plays the role of the mean offspring in the branching-process approximation of the local limit of the configuration model. In view of Condition~\ref{cond-regularity-graph}\ref{cond-regularity-graph-R3}, the condition $\nu>1$ amounts to the requirement $p \neq \delta_2$.


\subsection{Modified random walk}
\label{sec:modrw}

We define a \emph{modified random walk}, denoted by $(Y_t)_{t\in\N}$, as a random walk on a static random graph that at certain random times makes uniform jumps. Formally, we have a sequence $(J_t)_{t\in\N}$ of random variables adapted to a filtration $(\mathcal{F}_t)_{t\in\N}$, taking values in $\{0,1\}$ according to a pre-specified distribution on $\{0,1\}^\N$. For fixed $t\in\N$, $J_t$ is seen as the indicator of the event that the modified random walk makes a uniform jump at time $t$. The law of the modified random walk $(Y_t)_{t\in\N}$ on $\xi$ that starts from the initial half-edge $x\equiv Y_0$, which is adapted to $(\mathcal{F}_t)_{t\in\N}$, is given by the conditional probabilities
\begin{equation}
\label{MRWT}
\begin{aligned}
&\pmod_{x,\xi}(Y_t = z\mid Y_{t-1} = y, J_1 = j_1, \dots, J_t = j_t) \\
&= \pmod_{x,\xi}(Y_t = z\mid Y_{t-1} = y, J_t = j_t) = 
\begin{cases}	
P_{\xi}(y,z) & \text{if }j_t = 0, \\
\tfrac{1}{\sizeH} & \text{if }j_t = 1,
\end{cases}\quad t\in\N,
\end{aligned}
\end{equation}
with
\begin{equation}
\pmod_{x,\xi}(Y_0 = x) = 1.
\end{equation}
Note that, according to the definition, neither $(J_t)_{t\in\N}$ nor the pair $(Y_t,J_t)_{t\in\N}$ needs to be Markov, but $(Y_t)_{t\in\N}$ is Markov conditionally on a realisation of $(J_t)_{t\in\N}$.

Uniform jumps of the modified random walk can be rephrased in the following form. Let $Y'_t$ be a uniformly chosen half-edge, independent of the random walk path and the jump times. If $J_t = 1$, then we choose a uniform sibling of $Y'_t$, say $y$, and set $Y_t = y$. Since $Y'_t$ is uniform and one of its siblings is chosen uniformly at random, the resulting half-edge is distributed uniformly on $H$. Even though $Y^\prime_t$ is already a half-edge chosen uniformly at random, working with its sibling (which is also a half-edge chosen uniformly at random) will come in handy in the coupling argument in Section~\ref{sec:coupling}.

As an analogue of $\tau$, we define $\sigma$ to be the first time that the modified random walk makes a uniform jump, i.e.,
\begin{equation}
\label{eq:sigma}
\sigma \coloneqq \inf\{t\in\N: J_t = 1\}.
\end{equation}


\subsection{Coupling of modified and dynamically rewired random walk} 
\label{sec:coupling}

We couple the law $\prob_{x,\xi}(X_t\in\cdot)$ of the random walk on the dynamic random graph, with initial half-edge $x$ and initial configuration $\xi$, to the law $\pmod_{x,\xi}(Y_t\in\cdot)$ of the modified random walk. We want the coupled random walks to \emph{stick together as much as possible}. When the two random walks make different steps, we say that the coupling of the two random walks has \emph{failed}. Until the coupling fails, the times at which the random walk on the dynamically rewired graph makes a step over a previously rewired edge correspond to the times at which the modified random walk makes a uniform jump.

\begin{definition}[Coupling to a modified random walk]
\label{def:coupling}
Let $X_t$ be a non-backtracking random walk starting in the initial state $(x, \xi)$, where $x\in H$, $\xi \in \Conf_H$, and $Y_t$ be a modified random walk on $\xi$ starting in~$x$. First, define a sequence of auxiliary random sets $\left(A_t\right)_{t\in\N_0}$. Call $A_t$ the set of \emph{active} half-edges at time $t$. Let $A_0$ be the set consisting of the initial half-edge of the random walk and its siblings, i.e., $A_0 \coloneqq H_{v(x)}$. 

Define the coupling of the non-backtracking random walk $X_t$ and the modified random walk $Y_t$ at any time $t\in\N$ by the following rules:
\begin{enumerate}
\item If $\xi(X_{t-1})$ or any of its siblings belong to $A_{t-1}$, then declare the coupling as \textbf{failed}.
\label{coupling:fail3}
\item If $\deg_H(X_{t-1}) > \maxHdeg$ (recall that $n \coloneqq |V|$), then declare the coupling as \textbf{failed}. If Condition~\ref{cond-regularity-dynamics}\ref{cond-regularity-dynamics-D3} is not needed, then this rule is suspended (see Remark~\ref{rem:suspendfail4} below for further details).
\label{coupling:fail4}
\item If the coupling has not yet failed, then maximally couple the distribution of $I_t$, conditionally on the history of the random walk and the rewired edges seen by the random walk, to the distribution of $J_t$, conditionally on the values of the indicators $J_1,\dots,J_{t-1}$. The following three outcomes are possible:
\begin{enumerate}
\item 
If the coupling of the conditional distributions of $I_t$ and $J_t$ is successful and $I_t = J_t = 0$, then create $A_t$ as a union of $\xi(X_{t-1})$ and all its siblings with $A_{t-1}$. Let the random walk on a dynamic graph make a move and set $Y_t \coloneqq X_t$.
\item 
If the coupling of the conditional distributions of $I_t$ and $J_t$ is successful and $I_t = J_t = 1$, then maximally couple the distribution of $C_t(X_{t-1})$, i.e., the half-edge paired with $X_{t-1}$ in configuration $C_t$, conditionally on the history of the random walk and $I_t = 1$, to the distribution of $Y'_t$:
\begin{enumerate}
\item 
If the coupling of $C_t(X_{t-1})$ and $Y'_t$ is successful, and neither $C_t(X_{t-1})$ nor any of its siblings is already contained in $A_{t-1}$, then add $\xi(X_{t-1})$ and all its siblings, along with $C_t(X_{t-1})$ and all its siblings, to $A_{t-1}$ in order to obtain $A_t$. Phrased in symbols:
\begin{align}
\nonumber A_t \coloneqq A_{t-1} 
&\cup \xi(X_{t-1}) \cup \{ h \in H\colon\, h \sim \xi(X_{t-1}) \} \\ 
&\cup C_t(X_{t-1}) \cup \{ h \in H\colon\, h \sim C_t(X_{t-1}) \}.
\end{align}

Let the random walk on the dynamic graph make a move, and set $Y_t \coloneqq X_t$.
\item 
Otherwise, declare the coupling as \textbf{failed}. 
\label{coupling:fail1}
\end{enumerate}
\item 
If the coupling of the conditional distributions of $I_t$ and $J_t$ is not successful, namely if  $I_t \neq J_t$, then declare the coupling of the two random walks as \textbf{failed}. 
\label{coupling:fail2}
\end{enumerate}
\item 
If the coupling has failed let $X_t$ and $Y_t$ evolve independently.
\end{enumerate}
\end{definition}

\begin{remark}[Failure of the coupling after a high-degree half-edge is encountered]
\label{rem:suspendfail4}
In Lemma \ref{mainlemma} we will see that failure of the coupling as described in item~\ref{coupling:fail4} above is needed only when Condition~\ref{cond-regularity-dynamics}\ref{cond-regularity-dynamics-D3} comes into play. This will only happen for one of the three examples in Section~\ref{sec:ex}, namely, near-to-global. 
\end{remark}


\subsection{Failures in the coupling}
\label{sec:fail}

\begin{remark}[Possible failures]
\label{rem:failevent}
At each time $t\in\N$, the random walk and the coupled modified random walk try to avoid stepping on the active half-edges $A_{t-1}$. The coupling of these two random walks fails in four cases described in Definition~\ref{def:coupling}:
{\renewcommand{\theenumi}{\Roman{enumi}}
\renewcommand{\theenumii}{\Alph{enumii}}
\renewcommand{\labelenumii}{\theenumii.}
\begin{enumerate}
\item In step~\ref{coupling:fail1}:
\begin{enumerate}
\item 
if the coupling of $C_t(X_{t-1})$ and $Y'_t$ is not successful,
\label{rem:fail1a}
\item 
if the two random walks step over a half-edge in $A_{t-1}$. 
\label{rem:fail1b}
\end{enumerate}
\item 
In step~\ref{coupling:fail2}, if the coupling of $I_t$ and $J_t$ is not successful. 
\label{rem:fail2}
\item 
In step~\ref{coupling:fail3}, if the pair of $X_{t-1}$ in the starting configuration is already in $A_{t-1}$.
\label{rem:fail3}
\item
In step~\ref{coupling:fail4}, if the random walk encounters a half-edge $X_{t-1}$ with a high degree.
\label{rem:fail4}
\end{enumerate}
}
\noindent
Failure cases \ref{rem:fail1b} and \ref{rem:fail3} correspond to the situation in which the random walks do not have dynamically self-avoiding histories. Consequently, \emph{the random walks have dynamically self-avoiding histories before the coupling of the two random walks fails.} Failure case \ref{rem:fail1a} corresponds to the situation in which the conditional distribution of $C_t(X_{t-1})$ is too far from the uniform distribution in total variation distance. Failure case \ref{rem:fail2} corresponds to the situation in which the conditional distribution of the times at which the random walk on the dynamically rewired graph and the conditional distribution of the times at which the modified random walk makes uniform jumps are far from each other in total variation distance. Finally, failure case~\ref{rem:fail4} corresponds to the situation when during the graph exploration the random walk encounters a half-edge with an anomalously high degree.
\end{remark}

The next lemma states that these failure events are unlikely up to logarithmic times when Conditions~\ref{cond-regularity-graph} and \ref{cond-regularity-dynamics} hold for the random walk on the dynamically rewired random graph:

\begin{lemma}[Coupling estimates]
\label{mainlemma}
Suppose that $t = O(\log n)$, and that Conditions~\ref{cond-regularity-graph} and \ref{cond-regularity-dynamics} hold for the random walk on the dynamically rewired graph. For all $1 \leq s \leq t$ and all $T_s = \{s_1,\dots,s_r\}\subset[s-1]$, fix a sequence of half-edges
\begin{equation}
x^{T_s}_{[0,s-1]},\bar{x}^{T_s}_{[0,s-1]},\hat{x}^{T_s}_{[r]},\tilde{x}^{T_s}_{[r]}
\end{equation}
that describes a good dynamically self-avoiding history with respect to $T_s$ (see Definition~\ref{def:DSA}). Consider the modified random walk for which the jump distribution has conditional distribution
\begin{align}
\label{jumpdist}
&\pmod_{x,\xi}\big(J_s = 1\mid J_{s'} = 0\text{ for }s'\in[s-1]\setminus T_s, J_{s''} = 1\text{ for }s''\in T_s\big) \nn \\
&\qquad \coloneqq \prob_{x,\xi}\Big(I_s = 1\mid 
\mathsf{DSA}\big(T_s, x^{T_s}_{[0,s-1]}, \bar{x}^{T_s}_{[0,s-1]}, \hat{x}^{T_s}_{[r]}, \tilde{x}^{T_s}_{[r]}\big)\Big).
\end{align}
Then, whp in $x$ and $\xi$,
\begin{equation}
\|\prob_{x,\xi}(X_t\in\cdot) - \pmod_{x,\xi}(Y_t\in\cdot)\|_{\TV} = o_{\sss\prob}(1),
\end{equation}
and, with $\sigma$ as defined in \eqref{eq:sigma},
\begin{equation}
\prob_{x,\xi}(\tau > t) = \pmod_{x,\xi}(\sigma > t) + o_{\sss\prob}(1).
\end{equation}
\end{lemma}

\begin{remark}[Jump distribution of the modified random walk]
Observe that \eqref{jumpdist} describes the jump distribution of the modified random walk at any time $1 \leq s \leq t$ for any set of previous jump times $T_s$ in a \emph{non-anticipating} manner. If the sequence of half-edges in the event $\mathsf{DSA}(T_s, x^{T_s}_{[0,s-1]}, \bar{x}^{T_s}_{[0,s-1]}, \hat{x}^{T_s}_{[r]}, \tilde{x}^{ T_s}_{[r]})$ in the right-hand side of \eqref{jumpdist} is not compatible with the initial state $(x,\xi)$ (i.e., when the conditioning is on an event of probability zero), then we set the right-hand side of \eqref{jumpdist} equal to zero. The proof below uses an annealing argument in which the \enquote{mismatched} events play no role.
\end{remark}

\begin{proof}[Proof of Lemma~\ref{mainlemma}]
Let $\pcouple_{x,\xi}$ denote the law of the coupling of the two non-backtracking random walks described in Section~\ref{sec:modrw} with $X_0 = x$ and $C_0 = \xi$. Also, use $F \in \N$ to denote the time at which this coupling fails. Due to Condition~\ref{cond-regularity-graph}\ref{cond-regularity-graph-R3}, these random walks are always well-defined. Since the two random walks agree up to the time $F$, that is until the coupling fails, we have
\begin{equation}
\|\prob_{x,\xi}(X_t\in\cdot) - \pmod_{x,\xi}(Y_t\in\cdot)\|_{\TV} \leq \pcouple_{x,\xi}(F \leq t).
\end{equation}
So, in order to prove the claim it suffices to show that, whp in $x$ and $\xi$,
\begin{equation}
\pcouple_{x,\xi}(F \leq t) = o_{\sss\prob}(1).
\end{equation}

To achieve this, we use an annealing argument on the initial graph and the initial location. Recall that $\mu = U_H\times\UConf$, and let
\begin{equation}
\pcouple = \sum_{x,\xi}\mu(x,\xi)\,\pcouple_{x,\xi}.
\end{equation}
We will show that
\begin{equation}\label{annealedfail}
\pcouple(F \leq t) = o(1)
\end{equation}
by exploring the initial configuration using the paths of the random walk and its coupled modified random walk until the coupling fails at time~$F$.

\begin{enumerate}
\item At time $s=0$, choose a half-edge $x \in H$ uniformly at random. Set $X_0 = Y_0 = x$ and $A_0 = H_{v(x)}$, the subset of $H$ consisting of $x$ and its siblings.
\item At time $s\in\N$, first explore the half-edge to which $X_{s-1} = Y_{s-1}$ is paired in the initial configuration $\xi$, then let the coupled random walks evolve in accordance with Definition~\ref{def:coupling}, and update $A_s$ accordingly.
\end{enumerate}
This exploration process covers the part of the graph seen by the random walks, along with the parts affected by the rewiring at the positions of the random walks, and stops as soon as the coupling of the two random walks fails.

We will carry out the proof in a setting where Conditions~\ref{cond-regularity-dynamics}\ref{cond-regularity-dynamics-D1} and \ref{cond-regularity-dynamics-D2} hold. At the end of the proof we will briefly comment on the changes required when Condition~\ref{cond-regularity-dynamics}\ref{cond-regularity-dynamics-D3} comes into play.

Suppose that the coupling of the two random walks has not failed before time $s$. Failure at time $s$ can occur in the following three cases (see also Remark~\ref{rem:failevent}):
\begin{enumerate}
\item The coupling of $I_s$ and $J_s$ fails in step \ref{coupling:fail2} of Definition \ref{def:coupling}.
\label{case2}
\item The coupling of $C_s(X_{s-1})$ and $Y'_s$ fails in step \ref{coupling:fail1} of Definition \ref{def:coupling}.
\label{case3}
\item The random walks jointly step over a half-edge that lies in $A_{s-1}$ in either step~\ref{coupling:fail1} or step~\ref{coupling:fail3} of Definition \ref{def:coupling}.
\label{case4}
\end{enumerate}

For case~\ref{case2}, we note that, since the distribution of $J_t$ for the modified random walk is given by \eqref{jumpdist}, Condition~\ref{cond-regularity-dynamics}\ref{cond-regularity-dynamics-D1} implies that the probability of coupling failure is $o(1/\log n)$. 

For case~\ref{case3} we note that, by Remark~\ref{rem:failevent}, before the coupling of the two random walks fails, the random walk has a dynamically self-avoiding history. By Condition~\ref{cond-regularity-dynamics}\ref{cond-regularity-dynamics-D2}, the total variation distance between the conditional distribution of $C_s(X_{s-1})$ and the uniform distribution $U_H$ is $o(1/\log n)$. Since $Y'_s$ is also distributed uniformly on $H$, the probability of the event in case 2 is $o(1/\log n)$.

For case~\ref{case4}, we first need an upper bound on the size of $A_{s-1}$. Each time we explore the initial configuration, we add at most $\dmax$ half-edges to the set of active half-edges. In case a rewiring occurs, then we add at most $2\dmax$ half-edges to the set of active half-edges. This gives us the following crude bound:
\begin{equation}
|A_{s-1}| \leq 3s\dmax. \end{equation}
For a fail event in step~\ref{coupling:fail1}, we see that the probability that $C_s(X_{s-1})\in A_{s-1}$ is smaller than
\begin{equation}
\frac{|A_{s-1}|}{\sizeH}+o(1/\log n) \leq \frac{3s\dmax}{\sizeH}+o(1/\log n),
\end{equation}
since the random walk has a dynamically self-avoiding history before the coupling of the two random walks fails (see Remark~\ref{rem:failevent}), so the total variation distance between the conditional distribution of $C_s(X_{s-1})$ and the uniform distribution $U_H$ is $o(1/\log n)$, by Condition~\ref{cond-regularity-dynamics}\ref{cond-regularity-dynamics-D2}.

For a fail event in step~\ref{coupling:fail3}, we see that the probability that $C_0(X_{s-1})\in A_{s-1}$ is smaller than
\begin{equation}
\frac{|A_{s-1}|}{\sizeH-4s+4} \leq \frac{3s\dmax}{\sizeH-4s+4},
\end{equation}
since up to time $s$ we form at most $2s-2$ pairs in $C_0$, of which $s-1$ on the random walk path and an additional $s-1$ if rewiring occurs at each step up to time $s$.

The above estimates give us
\begin{equation}
\pcouple(F = s\mid F > s-1) \leq \frac{6s\dmax}{\sizeH-4s+4}+o\Big(\frac{1}{\log n}\Big).
\end{equation}
Taking a union bound up to time $t$, and using that by assumption $t = O(\log n)$, $\dmax = o(n/(\log n)^2)$ (Condition~\ref{cond-regularity-graph}\ref{cond-regularity-graph-R2}) and $\sizeH = \Theta(n)$ (Condition~\ref{cond-regularity-graph}\ref{cond-regularity-graph-R1}), we get
\begin{equation}
\pcouple(F \leq t) \leq \frac{3t(t+1)\dmax}{\sizeH-4t}+o(1) = o(1),
\end{equation}
which in turn implies that,
\begin{equation}
\pcouple_{x,\xi}(F\leq t) = o_{\sss\prob}(1).
\end{equation}

In case we rely on Condition~\ref{cond-regularity-dynamics}\ref{cond-regularity-dynamics-D3}, a fourth possible failure of the coupling shows up, namely, if the random walk encounters a half-edge of degree larger than $\maxHdeg$. The probability of this failure is $o(1/\log n)$ by Condition~\ref{cond-regularity-dynamics}\ref{cond-regularity-dynamics-D3}. The estimates for the other possible failures carry over, because if the coupling did not fail at some time $s$ due to a meeting with a high-degree half-edge, then the random walk path traced up to time $s$ is good and we can apply the same arguments as above.
\end{proof}


\subsection{Link between dynamic and static}
\label{sec:proofmainthm3}

In this section we prove Theorem~\ref{thm:main}. Consider the modified random walk in the statement of Lemma \ref{mainlemma} and sample uniform jump times up to time $t$. For any fixed $T = \{t_1,\dots,t_r\}\subset[t]$, we see that the modified random walk conditionally on the event $J(T) \coloneqq \{J_s = 0\text{ for } s\in[t]\setminus T, J_s = 1\text{ for } s\in T\}$ is a time-inhomogeneous Markov chain that makes random-walk steps at times $s\in[t]\setminus T$ and jumps to half-edges chosen uniformly at random at times $s\in T$.

Conditionally on $T \subset [t]$ being non-empty, it is obvious that at time $t$ the random walk on a graph satisfying Condition~\ref{cond-regularity-graph} is well-mixed for any starting $x \in H, \xi \in \Conf_H$ and so we claim that 
\begin{equation}
\pmod_{x,\xi}(Y_t\in\cdot\mid J(T)) = U_H(\cdot),
\end{equation}
and since $J(T)$, $\emptyset \neq T \subset [t]$ by definition implies $\sigma \leq t$, we also get
\begin{equation}
\label{pmodstopped}
\pmod_{x,\xi}(Y_t\in\cdot\mid\sigma\leq t) = U_H(\cdot).
\end{equation}
On the other hand, since the modified random walk up to time $t$ conditionally on the event $\{\sigma > t\}$ is the same as the random walk on the static graph, for any $x\in H$ and $\xi\in\Conf_H$, we have
\begin{equation}
\label{pmodunstopped}
\|\pmod_{x,\xi}(Y_t\in\cdot\mid\sigma>t) - U_H(\cdot)\|_{\TV} = \Dstat_{x,\xi}(t).
\end{equation}
Using the triangle inequality twice, we obtain
\begin{align}
\|\pmod_{x,\xi}(Y_t\in\cdot) - U_H(\cdot)\|_{\TV}\leq& \pmod_{x,\xi}(\sigma > t)\|\pmod_{x,\xi}(Y_t\in\cdot\mid\sigma>t) - U_H(\cdot)\|_{\TV} \nn \\
&+ \pmod_{x,\xi}(\sigma \leq t)\|\pmod_{x,\xi}(Y_t\in\cdot\mid\sigma\leq t) - U_H(\cdot)\|_{\TV},
\end{align}
and
\begin{align}
\|\pmod_{x,\xi}(Y_t\in\cdot) - U_H(\cdot)\|_{\TV}\geq& \pmod_{x,\xi}(\sigma > t)\|\pmod_{x,\xi}(Y_t\in\cdot\mid\sigma>t) - U_H(\cdot)\|_{\TV} \nn \\
&- \pmod_{x,\xi}(\sigma \leq t)\|\pmod_{x,\xi}(Y_t\in\cdot\mid\sigma\leq t) - U_H(\cdot)\|_{\TV}.
\end{align}
Inserting \eqref{pmodstopped} and \eqref{pmodunstopped}, we obtain
\begin{equation}
\|\pmod_{x,\xi}(Y_t\in\cdot) - U_H(\cdot)\|_{\TV} = \pmod_{x,\xi}(\sigma > t)\,\Dstat_{x,\xi}(t).
\end{equation}
Now using Lemma \ref{mainlemma}, we see that, whp in $x$ and $\xi$,
\begin{equation}
\Ddyn_{x,\xi}(t) = \prob_{x,\xi}(\tau > t)\,\Dstat_{x,\xi}(t) + o_{\sss\prob}(1),
\end{equation}
which concludes the proof of Theorem~\ref{thm:main}. \qed


\section{Examples of admissible dynamics}
\label{sec:ex}

In Section~\ref{sec:threerew} we introduce three choices of rewiring. In Sections~\ref{sec:locrrw}--\ref{sec:globrrw} we identify, for each of these choices, the scaling of the probability that the random walk does not step along a previously rewired edge, which settles Theorem~\ref{thm:notmain}. 

In Appendix~\ref{appA} we show that each of the three choices of rewiring leads to an irreducible and aperiodic joint Markov chain for the random walk and the random graph.


\subsection{Three choices of rewiring}
\label{sec:threerew}

We explore rewirings that fit into a larger scheme of random graph dynamics, namely, where the decision which edges to rewire depends on their distance to the current position of the random walk. 

\begin{definition}[Sets of edges to be rewired]
\label{def:edgesets}
Recall that the configuration $\xi$ is a pairing of all the half-edges (which induces a set of edges) and $H$ is the set of all half-edges. By abuse of notation, in Section~\ref{sec:intro:modelnotation} we introduced the expression $\{a,b\}\in\xi$, $a,b\in H$, to mean that the half-edges $a,b$ form an edge in the configuration $\xi$. For any $\xi \in \Conf_H$, $h \in H$ and $r_n \in \N$, define the following sets of edges:
\begin{equation}
\begin{aligned}
\localset{\xi}{h}
&:= \{ \{h,g\} \in \xi \},\\
\nearset{\xi}{h}{r_n}
&:= \left\{ \{k,l\}\in \xi\colon\,
\begin{array}{l}
k\in H \colon\, \prob(X_{t+\rho} = k \mid X_{t-1} = h,\, \xi\text{ fixed}) > 0 \text{ for } 0 \leq \rho < r_n,\\
l\in H \colon\, l=\xi(k)
\end{array}
\right\}.
\end{aligned}
\end{equation}
\end{definition}

\noindent
In words, $\mathrm{Local}_\xi(h)$ is the edge to which the half-edge $h$ belongs and $\mathrm{Near}_{\xi,r_n}(h)$ are the edges that can be reached in $r_n$ steps by the non-backtracking random walk when the graph is in configuration~$\xi$ (and is not evolving).

\begin{figure}[htbp]
\centering
\includegraphics[width=0.3\textwidth]{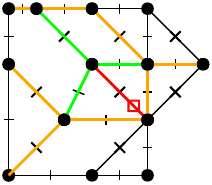}
\vspace{0.2cm}
\caption{\small Illustration of the sets in Definition~\ref{def:edgesets}. The red box denotes the current position of the random walk. The red edge forms the local set, which is also the near set with $r_n = 1$. The red and green edges form the near set with $r_n=2$. The red, green and orange edges form the near set with $r_n=3$.}
\label{fig:ball}
\end{figure}

With the above notation we can define the dynamics:

\begin{definition}[Random walk with $\left(K_t \right)$-to-$\left(L_t \right)$ rewiring]
\label{def:rewirings}
Recall that $X_{t-1}$ is the position of the random walk before the transition at time $t$ and $C_{t-1}$ is the configuration of the random graph before an update at time $t$. Let $\left(K_t \right)_{t\in\N}, \left(L_t\right)_{t\in\N}$ be sequences of sets of edges, which can be different at each time $t$. Define the random walk with $\left(K_t \right)$-to-$\left(L_t \right)$ rewiring as the following process:
\begin{enumerate}
\item 
At each time $t\in\N$, for each edge $e \in K_t$ draw a Bernoulli random variable $Z_t^{e}$ with parameter $\alpha_n$, independently of everything else.
\item 
\begin{enumerate}
\item If  $Z_t^{e} = 1$, then select edge $e$ for rewiring.
\item If $Z_t^{e} = 0$, then edge $e$ will not be rewired. 
\end{enumerate}
Write $R_t$ to denote the set of edges that get rewired at time~$t$.
\item
\begin{enumerate}
\item If $|R_t| \geq |L_t\setminus R_t|$, then break-up all the edges in $R_t \cup L_t$ into half-edges and re-pair them at random. More formally, pick $\frac12 |R_t \cup L_t|$ different half-edges (the half-edges forming $R_t \cup L_t$) and order them randomly. Also order randomly the half-edges not chosen in the previous step. The new pairing is generated by pairing the successive elements from the first and the second ordered sets described above. 
 
\item Otherwise, for every $e \in R_t$, choose $e^\prime \in L_t\setminus R_t$ uniformly at random without replacement. Denote the set of all edges $e^\prime$ chosen in the previous step by $R_t^\prime$. Break up $R_t$ into half-edges and order them randomly. Do the same with $R^\prime_t$. Just as in (a), the new pairing is given by the successive elements of the first and the second ordered set. 
\end{enumerate}

The new pairing of half-edges obtained in either (a) or (b) above is the new graph configuration~$C_t$.

\item 
The random walk moves from $X_{t-1}$ to $X_{t}$ on the evolved graph $C_{t}$.
\end{enumerate}
\end{definition}

\begin{remark}[Sets of edges generated from a configuration]\label{rem:edgeset}
When in the sequel we write $L_t \equiv \xi \in \Conf_H$, we mean that the set of edges $L_t$ is generated by the configuration $\xi$, which is a pairing of the entire set of half-edges $H$.
\end{remark}


\subsection{Local-to-global rewiring}
\label{sec:locrrw}

In this section we focus on a rewiring mechanism that is called local-to-global. Using the language of Definition \ref{def:rewirings}, this would be a rewiring with $K_t = \localset{C_{t-1}}{X_{t-1}}$ (see Definition~\ref{def:edgesets}) and $L_t \equiv C_{t-1}$ (see Remark~\ref{rem:edgeset}). Observe that the~set $K_t$ is explicitly dependent on the position of the random walk $X_{t-1}$ before the transition at time~$t$ occurs. For $\xi,\eta \in \Conf_H$ and $x \in H$, define
\begin{align}
\label{eq:QR}
Q^R_x(\xi,\eta) \coloneqq
\begin{cases}
\frac{1}{\sizeH-2} & \text{if }\xi(\eta(x))= \eta(\xi(x))\text{ and }|\xi\setminus\eta|= 2, \\
0 & \text{otherwise}.
\end{cases}
\end{align}
Then the transition matrix for the random graph from configuration $\xi$ to configuration $\eta$ when the random walk  is at position $x$ equals
\begin{equation}
\label{eq:Qdecomp}
Q_x(\xi,\eta) = (1-\alpha_n)I(\xi, \eta) + \alpha_n Q^R_x(\xi,\eta),
\end{equation}
where $I(\xi, \eta) = 1$ if $\eta = \xi$, and $I(\xi, \eta) = 0$ otherwise, i.e., $I$ is the identity matrix. The first term of \eqref{eq:QR} captures the situation when rewiring does not happen and the graph remains the same. On the other hand, the off-diagonal symmetric matrix $Q^R_x(\xi,\eta)$ in the second term represents the possible evolution of the graph by local-to-global rewiring. Note that the only possible transitions between graph states are those where the two configurations $\xi$ and $\eta$ differ in exactly two pairs of half-edges. The condition $\xi(\eta(x))= \eta(\xi(x))$ in \eqref{eq:QR} says that rewiring always happens at the position of the random walk. The value $\frac{1}{\sizeH - 2}$ comes from the fact that at time $t$ the rewiring mechanism can choose to pair the half-edge $X_t$ to any half-edge chosen randomly from $H \setminus \{X_{t-1}, C_{t-1}(X_{t-1})\}$, which is a set of size $|H| - 2$. 

Since $Q_x^R$ is symmetric for all $x\in H$, we see that the measure $\UConf$, defined by 
\begin{equation}
\UConf(\zeta) := \frac{1}{|\Conf_H|} \qquad \forall \zeta \in \Conf_H 
\end{equation}
is the stationary distribution for $Q_x^R$ for any $x\in H$. This implies that $\UConf$ is also the stationary distribution for $Q_x$ for all $x\in H$.

\begin{remark}[Symmetry of transition matrix for graph dynamics]
Local-to-global rewiring is one of the examples where the transition matrix is symmetric. Symmetry does not hold generally, even within the restricted class of \enquote{something-to-global} rewirings. Still, for such rewirings the transition matrices are always doubly stochastic. For more details see Appendix~\ref{appC}.
\end{remark}

Using this fact, we have the following result for the joint Markov chain:

\begin{proposition}[Stationary distribution]\label{prop:ltgstatdist}
For any $\alpha_n\in[0,1]$, $U_H\times\UConf$ is the stationary distribution for the random walk with local-to-global rewiring with parameter $\alpha_n$.
\end{proposition}

\begin{proof}
Recall from Section~\ref{ss:rw} that $P_\eta$ is the transition matrix for the non-backtracking random walk on the graph $\eta$. Since $U_H$ is stationary for $P_\eta$ for any $\eta\in\Conf_H$, and $\UConf$ is stationary for $Q_x$ for any $x\in H$, it follows that for any $y\in H$ and $\eta\in\Conf_H$,
\begin{equation}
\begin{aligned}
&\sum_{x\in H}\sum_{\xi\in\Conf_H}U_H(x)\UConf(\xi)\,\prob_{x,\xi}(X_1 = y, C_1 = \eta) \\
&\qquad =\sum_{x\in H}\sum_{\xi\in\Conf_H}U_H(x)\UConf(\xi)\,Q_x(\xi,\eta)P_\eta(x,y) \\
&\qquad =\sum_{x\in H}U_H(x)\,P_\eta(x,y)\sum_{\xi\in\Conf_H}\UConf(\xi)\,Q_x(\xi,\eta) \\
&\qquad =\UConf(\eta)\sum_{x\in H}U_H(x)P_\eta(x,y) = \UConf(\eta)U_H(y).
\end{aligned}
\end{equation}
\end{proof}

It is not obvious that the joint Markov chain is irreducible and aperiodic. In Appendix~\ref{appA} we show that this is nonetheless the case when $\alpha_n\in(0,1)$, and so the distribution of the joint Markov chain at time $t$ converges to $U_H\times\UConf$ as $t\to\infty$. An important implication is that the distribution of the random walk alone at time $t$ converges to $U_H$ as $t\to\infty$. Indeed, for any $x\in H$, $\xi\in\Conf_H$ and $t\in\N$, we have
\begin{equation}
\Ddyn_{x,\xi}(t) \leq \|\prob_{x,\xi}((X_t,C_t)\in\cdot)-U_H\times\UConf(\cdot)\|_{\TV},
\end{equation}
and since the right-hand side tends to $0$ as $t\to\infty$, $\Ddyn_{x,\xi}(t)$ also tends to $0$ as $t\to\infty$. On the other hand, this argument does not automatically imply that $\Ddyn_{x,\xi}(t)$ is non-increasing in $t$.

We are now ready to prove the scaling results stated in Theorem~\ref{thm:notmain}\ref{thm:notmain-A} and Corollary~\ref{locrwmixing}:

\begin{proof}
For fixed $t = O(\log n)$, fix some $T = \{t_1,\dots,t_r\}\subset[t-1]$ and some $x_{[0,t-1]}$, $\bar{x}_{[0,t-1]}$, $\hat{x}_{[r]}$ and $\tilde{x}_{[r]}$ that describe a dynamically self-avoiding history with respect to $T$. Conditionally on the event $\eventDSA$, $x_{t-1}$ cannot have been rewired before time $t$. Indeed, by construction the half-edges that are rewired before time $t$ are $x_{t_1-1},\dots,x_{t_r-1}$, $\bar{x}_{t_1-1},\dots,\bar{x}_{t_r-1}$, $\hat{x}_1,\dots,\hat{x}_r$ and $\tilde{x}_1,\dots,\tilde{x}_r$, while $x_{t-1}$ is not equal to any of these. So we have
\begin{align}
&\prob\big(I_t = 1\mid \eventDSA\big) \nn \\
&= \prob\big(Z_t^{\localset{C_{t-1}}{X_{t-1}}} = 1\mid \eventDSA\big) = \alpha_n.
\end{align}
Since this holds for any choice of $x_{[0,t-1]}$, $\bar{x}_{[0,t-1]}$, $\hat{x}_{[r]}$ and $\tilde{x}_{[r]}$, Condition~\ref{cond-regularity-dynamics}\ref{cond-regularity-dynamics-D1} holds with zero error. As a consequence, Condition~\ref{cond-regularity-dynamics}\ref{cond-regularity-dynamics-D1} is trivially satisfied. Moreover, $\prob(C_t(x_{t-1})\in\cdot\mid \eventDSA\cap\{I_t = 1\})$ is the uniform distribution on $H\setminus\{x_{t-1},$ $C_{t-1}(x_{t-1})\}$, because after rewiring the half-edge $x_{t-1}$ cannot end up being paired with itself or the half-edge it was paired with before. This gives
\begin{align}
\label{eq:2overH}
\Big\|\prob\big(C_t(x_{t-1})\in\cdot\mid \eventDSA \cap\{I_t = 1\}\big) - U_H(\cdot)\Big\|_{\TV} = \frac{2}{\sizeH}.
\end{align}
Since this holds for any choice of $x_{[0,t-1]}$, $\bar{x}_{[0,t-1]}$, $\hat{x}_{[r]}$ and $\tilde{x}_{[r]}$, Condition~\ref{cond-regularity-dynamics}\ref{cond-regularity-dynamics-D2} holds with error $O(1/n)$.  

On the other hand, the event $\{\tau = t\}$ is the same as the event $\{\min\{s\in\N: R_s = 1\} = t\}$, since when a rewiring occurs the random walk steps over a rewired edge with probability 1. This implies that, for any $x$ and $\xi$,
\begin{equation}
\prob_{x,\xi}(\tau > t \mid \eventSA{t}) = (1-\alpha_n)^t = \mathrm{e}^{-[1+o(1)]\,\alpha_n t},
\end{equation}
where $\eventSA{t}$ is the event that the random walk is self-avoiding until time $t$. The first equality comes from the requirement that none of the edges the random walk steps over until time $t$ gets rewired, the second equality uses that $\lim_{n\to\infty}\alpha_n = 0$. Since 
\begin{equation}
\label{SAWt}
\lim_{n\to\infty} \prob_{x,\xi}(\eventSA{t}) = 1 \quad \whp \text{ uniformly in } t = O(\log n),
\end{equation}
we obtain the scaling in Theorem~\ref{thm:notmain}\ref{thm:notmain-A}. (The proof of \eqref{SAWt} was given in \cite[Lemma 3.1]{AGHH20182} for global-to-global rewiring, but easily carries over to local-to-global and near-to-global rewiring.) Given Condition~\ref{cond-regularity-graph}\ref{cond-regularity-graph-R1}, we can use Corollary~\ref{maincor}, which combined with \eqref{c*def} yields Corollary~\ref{locrwmixing}.
\end{proof}


\subsection{Near-to-global rewiring}
\label{sec:nearrrw}

In this section we focus on near-to-global rewiring. In view of Definition~\ref{def:rewirings}, this is a rewiring with $K_t = \nearset{C_{t-1}}{X_{t-1}}{r_n}$ (recall Definition~\ref{def:edgesets}) and $L_t \equiv C_{t-1}$ (see Remark~\ref{rem:edgeset}) at any time $t$. Just like in the previous example, this is also a rewiring mechanism where the sets $K_t$ are dependent on the current position of the random walk.

The layout is the same as in the previous section, the main difference being the presence of the additional parameter $r_n$ that controls the size of the set of edges that are being considered for rewiring at each unit of time. We will see that this parameter controls the trichotomy.  We only consider $r_n = O(\log n)$, since the expected diameter of the configuration model is of order $\log n$ (see \eqref{CM:rad} and \cite{vdH2018, HHM2005}). For $r_n = o(\log n)$ the behaviour is dominated by the local properties of the graph dynamics and is similar to that for the local-to-global rewiring studied in Section~\ref{sec:locrrw}. On the other hand, once $r_n = \Theta(\log n)$ we get a significant contribution from a certain \enquote{boundary term} in the computation of the tail probability $\prob(\tau > t \mid \eventSA{t})$, and we find a behaviour that is more similar to the global-to-global rewiring studied in Section~\ref{sec:globrrw}.

First, we claim that the random walk is again irreducible and aperiodic:

\begin{proposition}[Irreducibility and aperiodicity]
Non-backtracking random walk with near-to-global rewiring is aperiodic and irreducible.
\end{proposition}

\begin{proof}
In Appendix~\ref{appA} we show that the joint Markov chain with local-to-global rewiring is irreducible and aperiodic. Since near-to-global rewiring admits all the transitions that are admitted for local-to-global rewiring, the proof carries over.
\end{proof}

Next, we claim that the stationary distribution is again uniform:

\begin{proposition}[Stationary distribution]
For any $\alpha_n \in[0,1]$ and $r_n = O(\log n)$, $U_H\times\UConf$ is the stationary distribution for the random walk with near-to-global rewiring with parameters $\alpha_n, r_n$.
\end{proposition}

\begin{proof}
Apply Proposition~\ref{prop:globuniform} to establish that $\UConf$ is stationary for the chosen graph dynamics. After that the rest of the proof carries over from Proposition~\ref{prop:ltgstatdist}. 
\end{proof}

We are now ready to prove Theorem~\ref{thm:notmain}\ref{thm:notmain-B} and Corollary~\ref{nearrwmixing}. First we settle Condition \ref{cond-regularity-dynamics}\ref{cond-regularity-dynamics-D2} for \emph{good} histories. After that we identify the asymptotics of $\prob(\tau > t \mid \eventSA{t})$ and settle Condition~\ref{cond-regularity-dynamics}\ref{cond-regularity-dynamics-D1} for \emph{good} histories. Both are tricky because they force us to investigate the possible occurrence of \emph{short-cuts} in the configuration. The key ingredient in the proof is that short-cuts are unlikely when $t = O(\log n)$ and $r_n \leq (1-\vep) \rho_\mathrm{max} \log n$ for some $\vep>0$, which requires the error term in Condition \ref{cond-regularity-dynamics}\ref{cond-regularity-dynamics-D1}. We finally settle Condition~\ref{cond-regularity-dynamics}\ref{cond-regularity-dynamics-D3}. At the end we put the pieces together and wrap up the proof.

\begin{proof}[Proof of Condition \ref{cond-regularity-dynamics}\ref{cond-regularity-dynamics-D2}.]
Because the rewiring is done with the global set, we have
\begin{align}
\prob\big(C_t(x_{t-1})\in\cdot\mid \eventDSA\cap\{I_t = 1\}\big) = U_{H\setminus\{x_{t-1}, C_{t-1}(x_{t-1})\}},
\end{align}
and, just as in \eqref{eq:2overH},
\begin{align}
\Big\|\prob\big(C_t(x_{t-1})\in\cdot\mid \eventDSA \cap\{I_t = 1\}\big) - U_H(\cdot)\Big\|_{\TV} = \frac{2}{\sizeH}.
\end{align}
Thus, Condition~\ref{cond-regularity-dynamics}\ref{cond-regularity-dynamics-D2} is satisfied.
\end{proof}

\begin{proof}[Identification of $\prob(\tau > t \mid \eventSA{t})$.]
On the event $\eventSA{t}$, for $1 \leq k <  l \leq t$, let $S^{r_n}_{kl}$ be the indicator of the event that there is a \emph{short-cut} of length $\leq r_n$ between the half-edges visited by the random walk \emph{at} times $k$ and $l$, i.e., a connection \emph{not} running along the path of the random walk itself. Abbreviate $\mathrm{SH}^{r_n}(t) = (S^{r_n}_{kl})_{1 \leq k < l \leq t}$. Then, for any $x,\xi$,
\begin{equation}
\label{neartauprob}
\begin{aligned}
&\prob_{x,\xi}(\tau > t \mid \eventSA{t},~\mathrm{SH}^{r_n}(t))\\ 
&\qquad = \prod_{i=1}^{(t-r_n)_+} (1-\alpha_n)^{\sum_{l=i+1}^{i+r_n} (1 + \sum_{k=1}^{i-1} S^{r_n}_{kl})} 
\prod_{i=(t-r_n)_+ + 1}^t (1-\alpha_n)^{\sum_{l=i+1}^{t} (1 + \sum_{k=1}^{i-1} S^{r_n}_{kl})}.
\end{aligned}
\end{equation}
This equality comes from the requirement that from time $1$ until time $(t-r_n)_+$ none of the $r_n$ half-edges on the \emph{future path} must be rewired, while from time $(t-r_n)_+ + 1$ until time $t$ none of the $r_n$ half-edges on the \emph{future path until time} $t$ must be rewired. Rewrite \eqref{neartauprob} as 
\begin{equation}
\label{neartauprobalt}
\begin{aligned}
\prob_{x,\xi}(\tau > t \mid \eventSA{t},~ \mathrm{SH}^{r_n}(t)) 
&= (1-\alpha_n)^{(t-r_n)_+ r_n + \tfrac12[t-(t-r_n)_+][t-(t-r_n)_+-1]}\\ 
&\qquad \times (1-\alpha_n)^{\chi^{r_n}(t)},
\end{aligned}
\end{equation}
with
\begin{equation}
\label{chi}
\begin{aligned}
\chi^{r_n}(t) &= \sum_{i=1}^{(t-r_n)_+} \sum_{l=i+1}^{i+r_n} \sum_{k=1}^{i-1} S^{r_n}_{kl}
+ \sum_{i=(t-r_n)_+ + 1}^t \, \sum_{l=i+1}^{t} \sum_{k=1}^{i-1} S^{r_n}_{kl}\\
&= \sum_{i=1}^t \bigg( \sum_{\substack{k \in (0,\, i)\\ l \in (i,\, t \wedge (i+r_n)]}}  S^{r_n}_{kl} \bigg).
\end{aligned}
\end{equation}
The first factor in \eqref{neartauprobalt} equals 
\begin{equation}
\label{factor}
\left\{\begin{array}{ll}
\exp\big(-[1+o(1)]\,\tfrac12\alpha_n t^2\big), &t \leq r_n,\\[0.2cm]
\exp\big(-[1+o(1)]\,\alpha_n [r_n(t-r_n)+\tfrac12 r_n^2]\big), &t \geq r_n,
\end{array}
\right.
\end{equation}
and produces the scaling in Theorem~\ref{thm:notmain}\ref{thm:notmain-B} (recall \eqref{SAWt}). We therefore need to show that the second factor in \eqref{neartauprobalt} is negligible. For this it suffices to show the following:

\begin{lemma}[Bound on number of short-cuts]
\label{lem-shortcuts}
Subject to Condition \ref{cond:secmom}, $\chi^{r_n}(t)=0$ $\whp$ uniformly in $t = O(\log n)$ and $r_n \leq (1-\vep) \rho_\mathrm{max} \log n$ for some $\vep>0$. 
\end{lemma} 

\begin{proof}
Recall that $\eventSA{t}$ is the event that the random walk is self-avoiding until time $t$. Consider the ball $B_t(x)$ of radius $t$ around the starting point $x$ of the random walk. Recall that, conditionally on $\eventSA{t}$, \eqref{degdef3} implies that the probability for the random walk to choose a $t$-step self-avoiding path consisting of half-edges $\vec{h} = (h_0, \ldots, h_{t-1})$ in $B_t(x)$ equals
\begin{equation}
\prod_{i=0}^{t-1} \frac{1}{\Hdeg(h_i)}.
\end{equation}
Condition on $\vec{h}$. Note that $\eventSA{t}$ is equivalent to the event that all half-edges in $\vec{h}$ are distinct, which we assume from now on.

It is helpful to distinguish between \emph{disjoint} short-cuts and \emph{non-disjoint} short-cuts. A disjoint short-cut between two half-edges $h_i$ and $h_j$ is a short-cut that does not use any of the other half-edges in $\vec{h}$. Not all short-cuts are disjoint. Indeed, a disjoint short-cut gives rise to other short-cuts that are counted in $\sum_{1 \leq k \leq t} S^{r_n}_{kt}$, which we call non-disjoint. For example, for $r_n\geq 2$, if there is a disjoint short-cut of one edge between $h_i$ and $h_{i+4}$, then there necessarily is a short-cut between $h_i$ and $h_{i+5}$ also. The point is that $\chi^{r_n}(t)=0$ precisely when there are no disjoint short-cuts. We must also bring the \emph{graph dynamics} into the picture.

We call a disjoint short-cut a disjoint $(s,i,j,k)$-short-cut when the $r_n$-neighbourhood of the random walk at time $s$ creates a disjoint short-cut consisting of $k$ edges between $h_i$ and $h_j$. This is only possible when $s\leq i\leq s+r_n$ and $k\leq r_n$, since otherwise $h_i$ would not be in the $r_n$-neighbourhood of the random walk at time $s$, and when $j>i+r_n$, since otherwise the path of $k$-edges would not be a short-cut.

We aim to show that, for $r_n\leq (1-\vep)\rho_{\max}\log n$ and $\vep>0$, the probability that there exists a disjoint $(s,i,j,k)$-short-cut vanishes as $n\to\infty$. To do so, we rely on the first-moment method. We make crucial use of the fact that the configuration model is the stationary distribution under our graph dynamics. This implies that, conditionally on $\vec{h}$, all other half-edges at time $s$ are paired uniformly at random, so that we can use configuration model estimates. Given $\vec{h}$, the expected number of disjoint $(s,i,j,k)$-short-cuts is bounded by (see \cite[Proposition 7.4]{vdH2018}) 
\begin{equation}
\label{eq:CMestimates}
[1+o(1)]\,\frac{\Hdeg(h_i) \Hdeg(h_j)}{\widehat{\ell_n}}\,\widehat{\nu}_n^{\,k-1},
\end{equation}
with
\begin{equation}
\label{eq:hatted}
\widehat{\ell_n} = \ell_n - O(\log n), \qquad \widehat{\nu}_n = \nu_n\,\frac{\ell_n}{\widehat{\ell_n}},
\end{equation}
where $\nu_n$ is the size-biased mean of the empirical degree distribution $p_n$ (recall \eqref{nudef}), $\ell_n$ is the sum of the degrees (= number of half edges), and the error term $o(1)$ is uniform in $k\leq C\log n$. The quantities in \eqref{eq:hatted} introduce corrections that come from the fact that, conditionally on $\vec{h}$, only a subset of size $\widehat{\ell_n}$ of the half-edges is randomly paired at time~$s$. Due to Condition \ref{cond:secmom}, the sum over $1 \leq k\leq r_n\leq (1-\vep)\rho_{\max}\log n$ of this expression is bounded by $(\max_{1 \leq i\leq t} \Hdeg(h_i))^2 n^{-\vep/2}$ for $n$ large enough. Thus, for $r_n\leq (1-\vep)\rho_{\max}\log{n}$, by a union bound over $1\leq i,j\leq t$, the probability that there exists a disjoint short-cut before time $t$ is bounded by
\begin{equation}
r_n t^2 \Big(\max_{1 \leq i \leq t} \Hdeg(h_i)\Big)^2 n^{-\vep/2}.
\end{equation}

Since $t=O(\log{n})$, we can use an annealing argument to show that, subject to Condition \ref{cond:secmom}, $\max_{1 \leq i \leq t} \Hdeg(h_i)\leq t^2$ whp. Indeed, let $\tilde{h}_i$ denote the half-edge to which $h_i$ is paired, so that $\Hdeg(h_{i+1})=\Hdeg(\tilde{h}_i)$. Then, the distribution of $\Hdeg(\tilde{h}_i)$ is the size-biased degree distribution minus 1. By Condition \ref{cond:secmom}, the mean of this size-biased distribution is uniformly bounded, so that by the Markov inequality the probability that $\Hdeg(h_{i+1})\geq B$ is at most $C/B$ for any $B>0$ and some $C<\infty$. Hence the probability that $\max_{1 \leq i \leq t} \Hdeg(h_i)> t^2$ is at most $Ct/t^2=o(1)$. 

Since $r_n,t=O(\log{n})$, we conclude that the probability that $\chi^{r_n}(t)>0$ is whp at most
\begin{equation}
r_n t^{6} n^{-\vep/2} = o(1),
\end{equation}
as required.
\end{proof}

We can now complete the identification of $\prob(\tau > t \mid \eventSA{t})$. By Lemma~\ref{lem-shortcuts}, $\prob_{x,\xi}(\tau > t \mid \eventSA{t},~\mathrm{SH}^{r_n}(t))$ is asymptotically equal to the expression in \eqref{factor} $\whp$, uniformly in $r_n \leq (1-\vep)\rho_\mathrm{max} \log n$ and $t = O(\log n)$. Taking the expectation w.r.t.\ $\mathrm{SH}^{r_n}(t)$, we get that the same is true for $\prob_{x,\xi}(\tau > t \mid \eventSA{t})$. Taking the expectation w.r.t.\ $x,\xi$ as well, we conclude that the same is true for $\prob(\tau > t \mid \eventSA{t})$, as required.
\end{proof}

\begin{proof}[Proof of Condition \ref{cond-regularity-dynamics}\ref{cond-regularity-dynamics-D1}.] For all paths that describe a dynamically self-avoiding history with respect to $T \subset [t-1]$, the probability that at time $t$ the random walk steps along a rewired edge is
\begin{align}\label{eq:cond35-D1-ntg}
\prob\big(I_t = 1 \mid \eventDSA\big) = \beta_{n,t}+\vep_{n,t}(\Tallx),
\end{align}
with (recall $t_r$ from \eqref{eq:def:T})
\eqn{
\beta_{n,t}=1-(1-\alpha_n)^{(t-t_r)\wedge r_n}
}
being the probability that the $t^{\rm th}$ edge is rewired when it is in the range of the random walk path, and $\vep_{n,t}(\Tallx)\geq 0$ is the contribution due to short-cuts. Note that $\beta_{n,t}$ is independent of $(\Tallx)$, so that to verify Condition \ref{cond-regularity-dynamics}\ref{cond-regularity-dynamics-D1}, we only need to bound $\vep_{n,t}(\Tallx)$.

To identify $\vep_{n,t}(\Tallx)$, we write
\eqan{
&\vep_{n,t}(\Tallx)\\
&\quad =(1-\beta_{n,t})\,\expec\Big[[1-(1-\alpha_n)^{\chi_*^{r_n}(t)}] \mid \eventDSA\Big]\nn,
}
with
\eqn{
\label{chi*}
\chi_*^{r_n}(t) = \sum_{k=1}^{(t-r_n)_+}S^{r_n}_{kt},
}
where $\vep_{n,t}(\Tallx)$ is the probability that the $t^{\rm th}$ edge is rewired due to a short-cut that puts it in the $r_n$-neighbourhood of the location of the random walk at some time $k<t-r_n$, but is not rewired due to a rewiring on the path of the random walk. The crux of the argument is to show that the event $\eventDSA$ affects a negligible amount of half-edges. After that we are in a situation where we can once again apply configuration model estimates, as in \eqref{eq:CMestimates}. 

The event $\eventDSA$ implies certain restrictions on the pairing of half-edges for every $s \in [0,t-1]$. These restrictions can be of two kinds: they can pair two half-edges with certainty or with a probability that depends on the fine details of the rewiring dynamics. In the near-to-global case these probabilities are generally close to $1$. Denote by $H_s$ the (partially) random set of half-edges that are paired by the event $\eventDSA$ at time$s$. The following observation is crucial:

\begin{lemma}[Random pairings outside $H_s$]
\label{lemma:outsideDSA}
Conditionally on $\eventDSA$, the half-edges in $H \setminus H_s$ are paired and rewired randomly at any time $s \in [0,t-1]$. Furthermore, $H_s \subseteq H^\prime$, where $H^\prime = (x_{[0,t-1]} \cup \bar{x}_{[0,t-1]} \cup \hat{x}_{[r]} \cup \tilde{x}_{[r]})$.
\end{lemma}

\begin{proof}
Since the graph is initially drawn according to the configuration model, and the configuration model is the stationary distribution of the graph dynamics, we see that on the set $H \setminus H_s$ the pairing is uniformly at random. Because the paired half-edges in $H_s$ are fixed, they do not affect the half-edges in $H \setminus H_s$. Let us clarify the possible restrictions implied by $\eventDSA$ at time $s$:
\begin{enumerate}
\item 
Edges already traversed by the random walk can get stuck in the configuration seen by the random walk. More formally, edges $\{X_q, C_q(X_q) \}$ with $q\leq s$ need not be a part of the near-set $\nearset{C_{p-1}}{X_{p-1}}{r_n}$ for any time $p \geq s$. This concerns half-edges in $x_{[0,t-1]}, \bar{x}_{[0,t-1]}$ and $\hat{x}_{[r]}$.
\item 
Edges that are traversed at time $q$ with $q>s$ and $q \notin T$ must not get rewired before the random walk crosses them. This concerns half-edges in $x_{[0,t-1]}$ and $\bar{x}_{[0,t-1]}$.
\item 
Edges that are traversed at time $q$ with $q>s$ and $q \in T$ can (but need not) get rewired before the random walk crosses them. If they get rewired just before the random walk crosses them and near-sets at times $<q$ do not contain $\tilde{x}_q$, then $\{x_q, \bar{x}_q \}$ and $\{\hat{x}_q, \tilde{x}_q\}$ must remain paired until time $q$. This concerns half-edges in $x_{[0,t-1]}, \bar{x}_{[0,t-1]}, \hat{x}_{[r]}$ and $\tilde{x}_{[r]}$. 
\end{enumerate}
Observe that only the edges that consist of half-edges in $x_{[0,t-1]}, \bar{x}_{[0,t-1]}, \hat{x}_{[r]}$, $\tilde{x}_{[r]}$ can be fixed. If we take the union of all these half-edges $H^\prime$, we get a crude upper estimate on $H_s$ that is valid for all $s \in [0,t-1]$.
\end{proof}

Next we estimate the number of half-edges that are influenced by the restrictions implied by $\eventDSA$:

\begin{lemma}[Estimate of influenced half-edges]
Conditionally on $\eventDSA$, $H_s$ satisfies the estimate $|H_s| = O(t)$ for any $s \in [0, t-1]$.
\end{lemma}

\begin{proof}
In view of Lemma~\ref{lemma:outsideDSA}, it suffices to bound the number of half-edges in the sequences $x_{[0,t-1]}, \bar{x}_{[0,t-1]}, \hat{x}_{[r]}$, $\tilde{x}_{[r]}$, namely, $|H^\prime| = O(t)$. The sequences $x_{[0,t-1]}$ and $\bar{x}_{[0,t-1]}$ each contain $t-1$ half-edges by definition. The numbers of half-edges in $\hat{x}_{[r]}$ and $\tilde{x}_{[r]}$ depend on the set $T \subset [t-1]$ of times when the random walk steps over a rewired edge. Pick $T = [t-1]$ to see that $\hat{x}_{[r]}$ and $\tilde{x}_{[r]}$ both contain at most $t-1$ half-edges. Summing the four contributions, we see that indeed $|H^\prime| = O(t)$.
\end{proof}

We are now ready to apply configuration-model estimates:
 
\begin{lemma}[Bound on number of short-cuts]
\label{lem-shortcutsalt}
Subject to Condition \ref{cond:secmom}, conditionally on $\eventDSA$, $\chi_*^{r_n}(t)=0$ $\whp$ uniformly in $t = O(\log n)$, $r_n \leq (1-\vep) \rho_\mathrm{max} \log n$ for some $\vep>0$, and $x_{[0,t-1]}, \bar{x}_{[0,t-1]}, \hat{x}_{[r]}$, $\tilde{x}_{[r]}$.
\end{lemma} 

\begin{proof}
Observe that $\chi_*^{r_n}(t) > 0$ implies the existence of a $(s,i,j,k)$-short-cut at some time $s \in [0, t-1]$. In Lemma~\ref{lem-shortcuts} we proved a result about rarity of these shortcuts where we assumed only Condition~\ref{cond:secmom}. The statement of the current lemma furthermore assumes that the event $\eventDSA$ occurs.

In Lemma~\ref{lemma:outsideDSA} we have shown that at any time $s \in [0, t-1]$ the conditioning on $\eventDSA$ only affects the pairing of some half-edges in $H_s$. In Lemma~\ref{lem-shortcutsalt} we gave an estimate of $|H_s|$ for any $s \in [0, t-1]$. These two results bring us into the same setting as we had in the proof of Lemma~\ref{lem-shortcuts}, namely, we see that configuration model estimates hold (recall \eqref{eq:CMestimates}). Therefore, by the same argument as above, given that $r_n, t = O(\log n)$, we once again claim that the probability of $\chi_*^{r_n}(t) > 0$ is at most
\begin{equation}
 r_n t^{2} \Big(\max_{1 \leq i \leq t-1} \Hdeg(x_i)\Big)^2 n^{-\vep/2}.
\end{equation}
Since Condition~\ref{cond-regularity-dynamics}\ref{cond-regularity-dynamics-D1} concerns sequences $x_{[0,t-1]}$ that are good, we have
\begin{equation}
\max_{1 \leq i \leq t-1} \Hdeg(x_i) \leq \maxHdeg,
\end{equation}
and so 
\begin{equation}
r_n t^{2} (\maxHdeg)^2 n^{-\vep/2} = o(1),
\end{equation}
as required. Note that $\chi_*^{r_n}(t) \leq \chi^{r_n}(t)$ (compare \eqref{chi} and \eqref{chi*}).
\end{proof}

Now we see that the contribution of the $\vep_{n,t}(\Tallx)$ term in \eqref{eq:cond35-D1-ntg} is $O(n^{-\vep/2})$ and therefore Condition~\ref{cond-regularity-dynamics}\ref{cond-regularity-dynamics-D1} holds.
\end{proof}

\begin{proof}[Proof of Condition~\ref{cond-regularity-dynamics}\ref{cond-regularity-dynamics-D3}]
Observe that, for $t = O(\log n)$,
\begin{equation}
\begin{aligned}
 \prob\left(X_{[0,t-1]}\text{ is bad} \right) 
&= \prob\left(\exists\,1\leq i \leq t\colon\, \Hdeg(X_i) > \maxHdeg\right)\\
&\leq \sum_{1 \leq i \leq t} \prob\left(\Hdeg(X_i) > \maxHdeg\right)\\
&\leq \sum_{1 \leq i \leq t} \frac{\expec[\Hdeg(X_i)]}{\maxHdeg}\\
&\leq t\,\frac{\max_{1\leq i \leq t} \expec[\Hdeg(X_i)]}{\maxHdeg} \\
&= O\left(\frac{1}{(\log{n})^{1+\vep}}\right),
\end{aligned}
\end{equation}
where we use that $\max_{1\leq i \leq t} \expec[\Hdeg(X_i)]$ is finite by Condition~\ref{cond:secmom}. 
Since $\vep>0$, Condition~\ref{cond-regularity-dynamics}\ref{cond-regularity-dynamics-D3} follows.
\end{proof}

\begin{proof}[Completion of the proof of Theorem~\ref{thm:notmain}\ref{thm:notmain-B} and Corollary~\ref{nearrwmixing}.] We already verified Condition \ref{cond-regularity-dynamics}, and have shown that $\prob(\tau > t \mid \eventSA{t})$ is asymptotically equal to the expression in \eqref{factor}. Furthermore, by \eqref{SAWt}, $\eventSA{t}$ occurs $\whp$, uniformly in $t = O(\log n)$. This completes the proof of Theorem~\ref{thm:notmain}\ref{thm:notmain-B}. Finally, given Condition~\ref{cond-regularity-graph}\ref{cond-regularity-graph-R1}, we can again use Corollary~\ref{maincor}, which combined with \eqref{c*def} yields Corollary~\ref{nearrwmixing}. 

\end{proof}


\subsection{Global-to-global rewiring}
\label{sec:globrrw}

In this section we focus on global-to-global rewiring. This choice was already explored in \cite{AGHH2018}, \cite{AGHH20182}, with the minor difference that in the present paper the parameter $\alpha_n$ is the probability that an edge gets rewired per unit of time, while in \cite{AGHH2018}, \cite{AGHH20182} it was the \emph{fraction} of edges that get rewired per unit of time. This difference has no impact on the scaling of the mixing times. Global-to-global rewiring corresponds to the choice $K_t = L_t \equiv C_{t-1} $ (see Remark~\ref{rem:edgeset}) for all $t$ in Definition~\ref{def:rewirings}. Unlike for the previous examples, now the rewiring is \emph{independent} of the position of the random walk, so the graph dynamics becomes Markovian.

As before, the use of Corollary~\ref{maincor} depends on Condition~\ref{cond-regularity-graph}\ref{cond-regularity-graph-R1}. The proof of Theorem~\ref{thm:notmain}\ref{thm:notmain-C} uses that for all $x$ and $\xi$,
\begin{equation}
\label{globtauprob}
\prob_{x,\xi}(\tau > t \mid \eventSA{t}) = \prod_{i=1}^t (1-\alpha_n)^{t-i} = \exp\big(-[1+o(1)]\,\tfrac12\alpha_n t^2\big).
\end{equation}
The first equality comes from the requirement that up to time $t$ each of the half-edges on the future path of the random walk up must not get rewired. We thus obtain the scaling in Theorem~\ref{thm:notmain}\ref{thm:notmain-C} (again recall \eqref{SAWt}). Given Condition~\ref{cond-regularity-graph}\ref{cond-regularity-graph-R1}, we can again use Corollary~\ref{maincor}, which combined with \eqref{c*def} yields Corollary~\ref{globalrwmixing}. 

Irreducibility and aperiodicity of the rewiring was settled in \cite{AGHH2018}. The fact that the stationary distribution is the configuration model is settled by Proposition~\ref{prop:globuniform}, in combination with an argument analogous to Proposition~\ref{prop:ltgstatdist}. It remains to establish Condition~\ref{cond-regularity-dynamics}.

\begin{proposition}[Regularity of graph dynamics for global-to-global rewiring]
Global-to-global rewiring satisfies the graph-dynamics regularity conditions formulated in Condition~\ref{cond-regularity-dynamics}.
\end{proposition}

\begin{proof}
Since any edge can get rewired at any time, we have
\begin{align}
&\prob\big(I_t = 1\mid \eventDSA\big) = \alpha_n^t,
\end{align}
where we use that the edge crossed at time $t$ has had exactly $t$ opportunities to get rewired. Since this holds for any choice of $x_{[0,t-1]}$, $\bar{x}_{[0,t-1]}$, $\hat{x}_{[r]}$ and $\tilde{x}_{[r]}$, Conditions~\ref{cond-regularity-dynamics}\ref{cond-regularity-dynamics-D1} follows with zero error. Moreover, since a half-edge can get rewired to any half-edge except itself and its current pair, we know that $\prob(C_t(x_{t-1})\in\cdot\mid\eventDSA\allowbreak\cap\{I_t = 1\})$ is the uniform distribution on $H\setminus\{x_{t-1},\allowbreak C_{t-1}(x_{t-1})\}$, which gives
\begin{align}
\label{tvgtg}
\Big\|\prob\big(C_t(x_{t-1})\in\cdot\mid \eventDSA \cap\{I_t = 1\}\big) - U_H(\cdot)\Big\|_{\TV} = \frac{2}{\sizeH}.
\end{align}
Since \eqref{tvgtg} holds for any choice of $x_{[0,t-1]}$, $\bar{x}_{[0,t-1]}$, $\hat{x}_{[r]}$, $\tilde{x}_{[r]}$, Condition~\ref{cond-regularity-dynamics}\ref{cond-regularity-dynamics-D2} also follows. 
\end{proof}

\begin{remark}[Comparison with previous results]
The proof in \cite{AGHH2018} and \cite{AGHH20182} required a condition analogous to Condition~\ref{cond:secmom}, while in the present proof this is no longer needed.
\end{remark}


\appendix


\section{Irreducibility and aperiodicity}
\label{appA}

In this section we show that the random walk with \emph{local-to-global} rewiring is irreducible and aperiodic. This ensures that the total variation distance $\D_{x,\xi}(t)$ converges to $0$ as $t\to\infty$ for fixed $x\in H$, $\xi\in \Conf_H$ and $\alpha_n\in(0,1)$. Our proof builds on the proof of irreducibility of the switch chain on multigraphs given in \cite{EH1979}.

\begin{proposition}[Irreducible and aperiodic] 
The random walk with local-to-global rewiring $(X_t,C_t)_{t\in\N}$ (see Section~\ref{sec:locrrw}) is irreducible and aperiodic for any initial state $(x,\xi) \in H\times\Conf_H$ and any choice of $\alpha_n\in(0,1)$.
\end{proposition}

\begin{proof}
Let $V = \{v_1,\dots,v_n\}$ and assume that $\deg(v_1)\leq\deg(v_2)\leq\dots\leq\deg(v_n)$. Identify the set of half-edges $H$ with $[\sizeH] = \{1,\dots,\sizeH\}$, such that the half-edges $1,\dots,\allowbreak\deg(v_1)$ are associated to $v_1$, the half-edges $\deg(v_1)+1,\dots,\deg(v_1)+\deg(v_2)$ to $v_2$, and so on. Let $v_1',\dots,v_{2k}'\in V$ be the odd-degree vertices. We fix a configuration $\xi_0\in\Conf_H$ such that each vertex has the maximum number of self-loops, i.e., each vertex $v\in V$ with even degree has $\tfrac12\deg(v)$ self-loops, each vertex $v\in V$ with odd degree has $\tfrac12(\deg(v)-1)$ self-loops, and there is exactly one edge between every pair of odd-degree vertices $v_{2i-1}', v_{2i}'$ for $i = 1,\dots,k$ (see Figure~\ref{fig:xizero}). We will show that the pair $(1,\xi_0)\in H\times\Conf_H$ is accessible from any pair $(x,\xi)\in H\times\Conf_H$ by allowed moves for the random walk with local rewiring.

\begin{figure}[htbp]
\vspace{0.2cm}
\centering
\includegraphics[width=0.75\textwidth]{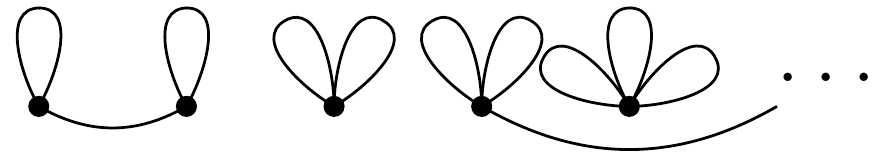}
\vspace{-0.2cm}
\caption{\small The configuration $\xi_0$.}
\label{fig:xizero}
\end{figure}

First we show that, for any $x\in H$, $(1,\xi_0)$ is accessible from $(x,\xi_0)$, by considering two different scenarios:
\begin{enumerate}
\item Suppose that $x$ is on a self-loop and $\xi_0(x) = x'$. We first move to $(1,\xi_1)$ from $(1,\xi_0)$ by rewiring the half-edges $x,x',1,2$ where $\xi_0$ and $\xi_1$ agree on all the edges except that $\xi_1(1) = x'$ and $\xi_1(2) = x$. After that we again move to $(1,\xi_0)$ from $(1,\xi_1)$ by rewiring $1,2,x,x'$ (see Figure~\ref{fig:xizeromove}). 
\item Suppose that $x$ is not on a self-loop, i.e., it is on an edge between two odd-degree vertices. We first move to $(x',\xi_0)$ without rewiring, where $x'\in H$ is on a self-loop. After that we apply the procedure in item 1 to $(x',\xi_0)$.
\end{enumerate}

\begin{figure}[htbp]
\centering
\includegraphics[width=0.75\textwidth]{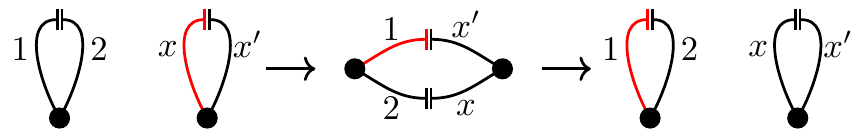}
\caption{\small Move from half-edge $x$ on a self-loop to half-edge $1$ in $\xi_0$. The red colour indicates the position of the walk.}
\label{fig:xizeromove}
\end{figure}

Next, we show that for any $(x,\xi)\in H\times\Conf_H$ with $\xi\neq\xi_0$ we have access from $(x,\xi)$ to $(y,\xi_0)$, for some $y\in H$. To do this, we show that we can move from $(x,\xi)$ to some $(y,\eta)\in H\times\Conf_H$ such that the configuration $\eta$ has more edges in common with $\xi_0$ than $\xi$ has, i.e., $|\xi\cap\xi_0|<|\eta\cap\xi_0|$, by considering the two scenarios:
\begin{enumerate}
\item 
Suppose that  $x$ is on an edge that is not in $\xi_0$, i.e., $\xi(x)\neq\xi_0(x)$. Then we move to $(y,\eta)$ by rewiring the half-edges $x,\xi(x),\xi_0(x),\xi(\xi_0(x))$, where $\xi$ and $\eta$ agree on all the edges except that $\eta(x) = \xi_0(x)$ and $\eta(\xi(x)) = \xi(\xi_0(x))$ and $y\sim\xi_0(x)$. Since $\eta(x) = \xi_0(x)$, we have $|\xi\cap\xi_0|\leq|\eta\cap\xi_0|-1$.
\item 
Suppose that $x$ is on an edge that is in $\xi_0$, i.e., $\xi(x) = \xi_0(x)$. Let $y\in H$ be a half-edge such that $\xi(y)\neq\xi_0(y)$, $\xi(x) = x'$ and $\xi(y) = y'$. Since $\deg(v)\geq 2$ for all $v\in V$, in the graph given by $\xi$ there is a cycle of edges $\{y,y'\},\{y_1,y_1'\},\dots,\{y_K,y_K'\}$ with $v(y') = v(y_1)$, $v(y_K') = v(y)$ and $v(y_i') = v(y_{i+1})$ for $i = 1,\dots,K-1$. Let $\eta\in\Conf_H$ be the configuration that agrees with $\xi$ on all the edges except that $\eta(x) = y'$ and $\eta(y) = x'$, so that the edges $\{y_1,y_1'\},\dots,\{y_K,y_K'\}$ are present in $\eta$ as well as in $\xi$. First we move from $(x,\xi)$ to $(y_1,\eta)$ by rewiring $x,x',y,y'$. Then we make $K$ moves, from $(y_i,\eta)$ to $(y_{i+1},\eta)$ for $i=1,\dots,K$, where $y_{K+1} = y$ without rewiring. After that we move from $(y,\eta)$ to $(y_1,\xi)$ by rewiring $x,x',y,y'$, and finally we traverse the cycle again without rewiring to reach $(y,\xi)$ from $(y_1,\xi)$ (see Figure~\ref{fig:cyclemove}). Now $y$ is on an edge that is not in $\xi_0$, so by applying the procedure in item 1 we can increase the number of edges we have in common with $\xi_0$.
\end{enumerate} 
By applying these procedures, we can reduce the number of edges that are not in $\xi_0$. So, we can go from any $(x,\xi)\in H\times\Conf_H$ to $(y,\xi_0)$ for some $y\in H$, and then apply the above procedure to reach $(1,\xi_0)$.

\begin{figure}[htbp]
\centering
\includegraphics[width=0.75\textwidth]{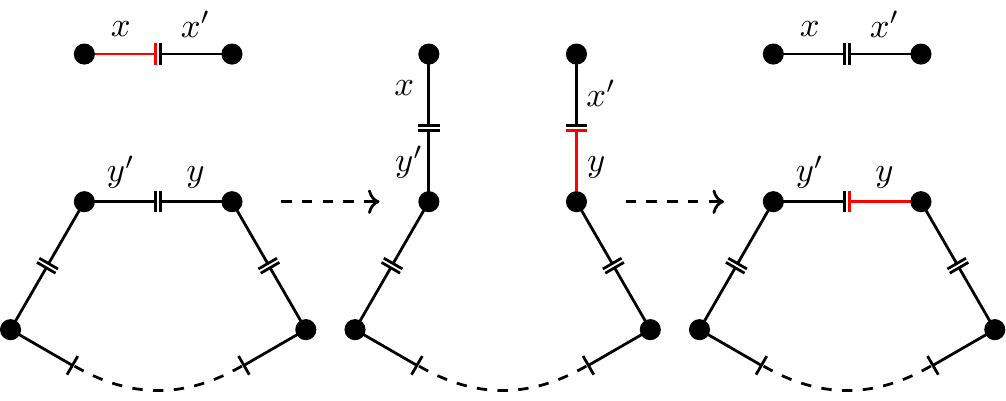}
\caption{\small Moving from $(x,\xi)$ to $(y,\eta)$ by using a cycle. The red colour indicates the position of the walk.}
\label{fig:cyclemove}
\end{figure}

To show that we can access an arbitrary state $(x,\xi)$ from $(1,\xi_0)$, we first note that we can access $(y,\xi_0)$, for any $y$, from $(1,\xi_0)$ by relabelling the half-edges and using the first argument above. Then we see that we can access $(x,\xi)$ from $(y,\xi_0)$ for any $y$ by using the above strategy of reducing the edges and using the cycles to move around. Hence, the Markov chain is irreducible. Since, by traversing the self-loop without rewiring, we can reach $(1,\xi_0)$ from itself in one step, we see that the Markov chain is also aperiodic.
\end{proof}


\section{Cut-off without dynamics}
\label{appB}

In order to use the results of~\cite{BHS2017} we need to assume the conditions stated there:

\begin{condition}[Additional regularity of degrees]
\label{cond-regularity-graph2}
$\mbox{}$
\begin{itemize}
\customitem{(R1**)}\label{cond-regularity-graph2-R1}
$\dmax \coloneqq \max_{v\in V}\deg(v)  = n^{o_(1)}$ as $n\to\infty$.
\customitem{(R2**)}\label{cond-regularity-graph2-R2}
\begin{equation*}
\frac{\lambda_2}{\lambda_1^3} = \omega\left(\frac{(\log\log\sizeH)^2}{\log\sizeH}\right),
\qquad \frac{\lambda_2^{3/2}}{\lambda_3\sqrt{\lambda_1}} 
= \omega\left(\frac{1}{\sqrt{\log\sizeH}}\right),\quad n\to\infty,
\end{equation*}
where
\begin{equation*}
\lambda_1 \coloneqq \frac{1}{\sizeH} \sum_{z\in H} \log(\Hdeg(z)), \qquad 
\lambda_m \coloneqq \frac{1}{\sizeH} \sum_{z\in H} |\log(\Hdeg(z))-\lambda_1|^m, \quad m = 2,3.
\end{equation*}
\customitem{(R3**)}\label{cond-regularity-graph2-R3}
$\deg(v)\geq 3$ for all $v\in V$.
\end{itemize}
\end{condition}

\noindent
Conditions~\ref{cond-regularity-graph2}\ref{cond-regularity-graph2-R1} and \ref{cond-regularity-graph2-R2} are technical and proof-generated. It might be possible to relax them via a truncation argument \cite{BHSprep}. Condition~\ref{cond-regularity-graph2}\ref{cond-regularity-graph2-R3} ensures that the random walk does not behave deterministically and that the configuration model is connected whp. Note that \ref{cond-regularity-graph2-R1} and \ref{cond-regularity-graph2-R3} are considerably more stringent than \ref{cond-regularity-graph-R2} and \ref{cond-regularity-graph-R3} in Condition~\ref{cond-regularity-graph}.

As shown in \cite{BHS2017}, the~following holds:

\begin{theorem}[Scaling of static mixing time]     
\label{thm:scalstat}
Subject to Condition~\ref{cond-regularity-graph2},
\begin{equation}
\Dstat_{x,\xi}(t) = 
\begin{cases}
1-o_{\sss\prob}(1), 
&\text{ if } \limsup_{n\to\infty}t/\tmixstat < 1, \\
o_{\sss\prob}(1),
&\text{ if } \liminf_{n\to\infty}t/\tmixstat > 1,
\end{cases}
\end{equation}
where
\begin{align}
\tmixstat \coloneqq [1+o_{\sss\prob}(1)]\,\cstat\log n \qquad \forall\, \varepsilon \in (0,1),
\end{align}
with $1/\cstat \coloneqq \frac{1}{\sizeH} \sum_{z\in H} \log(\Hdeg(z)) \in (0,\infty)$.
\end{theorem}

\noindent
If, in addition,
\begin{equation}
\label{c*def}
\lim_{n\to\infty} \cstat = c_*,
\end{equation}
then Theorem~\ref{thm:scalstat} yields Theorem~\ref{thm:stat}.

\begin{remark}
We are aware of the fact that Condition~\ref{cond-regularity-graph2} is in~\cite{BHS2017} used to prove a much stronger statement than Theorem~\ref{thm:scalstat} that is related to an exact computation of the cut-off window. Therefore, if we are interested only in proving Theorem~\ref{thm:scalstat}, weaker conditions might suffice.
\end{remark}

Combining Theorem~\ref{thm:main} and Theorem~\ref{thm:scalstat}, we obtain the following corollary:

\begin{corollary}[Link between static and dynamic]
\label{maincor}
Suppose $t = O(\log n)$. Subject to Condition~\ref{cond-regularity-graph}\ref{cond-regularity-graph-R1}, Condition~\ref{cond-regularity-dynamics} and Condition~\ref{cond-regularity-graph2}, the following holds whp in $x$ and $\xi$:
\begin{equation}
\label{maincorrequ1}
\Ddyn_{x,\xi}\big(t) =
\begin{cases}  
\prob_{x,\xi}(\tau > t)+o_{\sss\prob}(1), 
&\text{ if }\limsup_{n\to\infty}t/\tmixstat < 1,\\[0.2cm]
o_{\sss\prob}(1),
&\text{ if }\liminf_{n\to\infty}t/\tmixstat > 1.
\end{cases}
\end{equation}
\end{corollary}


\section{Transition matrix for \texorpdfstring{$(K_t)$-to-$(L_t)$}{(K\_t)-to-(L\_t)} rewiring}
\label{appC}

Recall Definition~\ref{def:rewirings}, where we have introduced the general class of rewirings considered in this paper. In this appendix we provide a general expression for the transition matrix of the graph dynamics. Furthermore, we explore the conditions required for this transition matrix to be doubly stochastic. 

\begin{proposition}[Transition matrix for $(K_t)$-to-$(L_t)$ rewiring]
The transition matrix for the rewirings in Definition~\ref{def:rewirings} is
\begin{equation}
\label{eq:gentransmatrix}
Q_{K_t \to L_t} = \sum_{k=0}^{|K_t|} (1-\alpha_n)^{|K_t|-k} (\alpha_n)^k  
\sum_{\{e_1,\ldots, e_k\} \in K_t} Q^{K_t \to L_t}_{\{e_1,\ldots, e_k\}}. 
\end{equation}
The matrix element $Q_X^{K_t \to L_t}(\eta, \xi)$ that represents the rewiring of the edges in the set $X$ that realises the transition from graph state $\eta$ to graph state $\xi$ is given by
\newlength{\eqparboxwidth}
\settowidth{\eqparboxwidth}{by rewiring all edges in $X$,}
\begin{equation}
Q_X^{K_t \to L_t}(\eta, \xi) = 
\begin{cases}\displaystyle
\frac{1}{2^{|X|}\prod \limits_{i=0}^{|X|-1} \left( |L_t| - |L_t \cap K_t| - i \right)}
&\parbox{\eqparboxwidth}{if $\xi$ is accessible from $\eta$\\ by rewiring all edges in $X$,}\\
0 &\text{otherwise.}
\end{cases}
\end{equation}
\end{proposition}

Observe that the matrix given by \eqref{eq:gentransmatrix} is a sum of multiple terms. Let us explain the meaning of these terms through the example of the general term
\begin{equation}
(1-\alpha_n)^{|K_t|-k} \alpha_n^k  \sum \limits_{\{e_1,\ldots, e_k\}\in K_t} Q^{K_t\rightarrow L_t}_{\{ e_1,\ldots, e_k\} }.
\end{equation}
First, the factor $(1-\alpha_n)^{|K_t|-k}$ represents the probability of $|K_t|-k$ edges not getting rewired, and its counterpart $\alpha_n^{k}$ represents the probability of $k$ edges getting rewired. The sum runs over all $k$-tuples from $K_t$, and the matrix $Q^{K_t\rightarrow L_t}_{\{ e_1,\ldots, e_k\}}$ represents the possible rewiring of the $k$-tuples of edges we are summing over.

The Markov chain transition matrix must be stochastic. Let us check this by an explicit computation. Take an arbitrary graph state $\eta$. In the row that lists the probabilities of all the possible transitions from $\eta$, we get the following contributions:
\begin{equation}
\begin{aligned}
\sum \limits_{\xi \in \Conf_H}Q_{K_t\rightarrow L_t} (\eta,\xi)
&= \sum_{k=0}^{|K_t|} (1-\alpha_n)^{|K_t|-k} \alpha_n^k \binom{|K_t|}{k} 
\frac{2^{k}\prod \limits_{i=0}^{k-1} (|L_t| - |L_t \cap K_t| - i)}{2^{k}\prod \limits_{i=0}^{k-1} (|L_t| - |L_t \cap K_t| - i)} \\
&= \sum_{k=0}^{|K_t|} (1-\alpha_n)^{|K_t|-k} \alpha_n^k \binom{|K_t|}{k}  \\
& = \left[(1-\alpha_n) + \alpha_n\right]^{|K_t|} = 1.
\end{aligned}
\end{equation}
The combinatorial factor $\binom{|K_t|}{k}$ counts the different ways of choosing $k$-tuples from $K_t$. Since the entries in $Q_X^{K_t\rightarrow L_t}(\eta, \xi)$ are chosen to be the reciprocal of the number of accessible states, it is not surprising that they sum up to 1. The factor $2^k$ comes from the ability to break up an edge into two ordered sets of half-edges.

Observe that the matrix defined by \eqref{eq:gentransmatrix} has a \enquote{binomial} structure, but that it is \emph{not} of the form
\begin{equation}
\label{eq:wrongtransmatrix}
Q_X^{K_t\rightarrow L_t}(\eta, \xi) = \prod \limits_{e\in X} \left( (1-\alpha_n) I 
+ \alpha_n  Q^{K_t\rightarrow L_t}_{\{e\} } \right).
\end{equation}
Clearly, \eqref{eq:wrongtransmatrix} would be correct if we would draw $e^\prime \in R^\prime_t$ in Definition~\ref{def:rewirings} \emph{with replacement}, when the state space for the rewiring of $|X|$ edges would have size $2^{|X|} \prod_{i=1}^{|X|} (|L_t| - |L_t\cap K_t|)$. In the current setting, where we draw $e^\prime$ \emph{without replacement}, the state space for the rewiring of $|X|$ edges is smaller, namely, size $2^{|X|}\prod_{i=0}^{|X|-1} (|L_t| - |L_t \cap K_t| - i)$, due to the removal of already drawn edges.

While we have seen that the transition matrix is stochastic, it is doubly stochastic only subject to additional conditions. For the purpose of this paper we need the following fact:

\begin{proposition}[Double stochasticity of $(K_t)$-to-global rewiring transition matrix]
\label{prop:globuniform}
The transition matrix given in \eqref{eq:gentransmatrix} is doubly stochastic for $L_t \equiv \xi$, in the sense that edges in $L_t$ are generated by pairing $\xi$ of the whole set of half-edges $H$ (recall Remark~\ref{rem:edgeset}).
\end{proposition}

\begin{proof}
The proof is by explicit computation. Choose an arbitrary graph state $\xi$ and count the contributions to the sum over the row corresponding to transitions leading to $\xi$:
\begin{equation}
\sum \limits_{\xi \in \Conf_H}Q_{K_t\rightarrow L_t} (\eta,\xi) 
= \sum_{k=0}^{|K_t|} (1-\alpha_n)^{|K_t|-k} \alpha_n^k \binom{|K_t|}{k} \frac{\prod \limits_{i=0}^{k-1} (|H| - 2|K_t| - 2i)}
{2^{k}\prod \limits_{i=0}^{k-1} (|L_t| - |L_t \cap K_t| - i)}.
\end{equation}
The term $(|H| - 1 - (2|K_t|-1) - 2i)$ is based on the following observation. We are counting possible pairs of half-edges where we see a difference in $\xi$ compared to $\eta$. This way we get the whole set of half-edges $|H|$, without the considered half-edge itself and without all but one half-edge in $K_t$. Rewiring cannot create an edge between half-edges that gave rise to $K_t$, and the term $-1$ arises from the one half-edge from $K_t$ the considered half-edge is paired with in $\xi$. The term $-2i$ again arises because we are drawing without replacement. Now observe that $2|L_t| = |H|$ and $L_t \cap K_t = K_t$. Apply the binomial theorem to get the claim.
\end{proof}


\bibliographystyle{abbrv}
\bibliography{references}{}


\end{document}